\documentclass{amsart}
\usepackage{amsmath,amssymb,hyperref,mathrsfs,graphicx}
\usepackage[utf8x]{inputenc}
\usepackage{amsmath}
\usepackage{amsfonts}
\usepackage{latexsym}
\usepackage{amsthm}
\usepackage{amssymb,amscd}
\usepackage{xargs}
\usepackage{tikz}
\usepackage{graphicx}
\usepackage{color}
\usepackage{enumitem}

\usepackage[colorinlistoftodos,prependcaption]{todonotes}
\presetkeys%
{todonotes}%
{inline,backgroundcolor=white}{}
\tikzset{/tikz/notestyleraw/.append style={text=black}}

\newtheorem{thm}{Theorem}[section]

\newtheorem{lem}[thm]{Lemma}

\newtheorem{prop}[thm]{Proposition}

\newtheorem{rmk}[thm]{Remark}
\newcommand{\be}{\begin{eqnarray}}
\newcommand{\ee}{\end{eqnarray}}
\newcommand{\beal}{\begin{aligned}}
\newcommand{\enal}{\end{aligned}}

\newcommand{\eps}{\varepsilon}

\newcommand{\lb}{\lambda}

\newcommand{\T}{\mathbb{T}}
\newcommand{\R}{\mathbb{R}}
\newcommand{\Q}{\mathbb{Q}}

\newcommand{\Z}{\mathbb{Z}}
\newcommand{\N}{\mathbb{N}}

\newcommand{\Lb}{\Lambda}

\newcommand{\Om}{\Omega}
\newcommand{\Dt}{\Delta}

\newcommand{\dt}{\delta}

\newcommand{\cB}{\mathcal{B}}

\newcommand{\cD}{\mathcal{D}}

\newcommand{\cS}{\mathcal{S}}

\newcommand{\cO}{\mathcal{O}}
\newcommand{\cP}{\mathcal{P}}
\newcommand{\cG}{\mathcal{G}}
\newcommand{\cH}{\mathcal{H}}
\newcommand{\cZ}{\mathcal{Z}}
\newcommand{\cQ}{\mathcal{Q}}

\newcommand{\cU}{\mathcal{U}}

\newcommand{\wh}{\widehat }
\newcommand{\wt}{\widetilde }
\newcommand{\ol}{\overline }

\title{Oscillatory orbits in the Restricted Planar 4 Body Problem}

\address{*\ Mathematics Institute, Academy of Mathematics and systems science
Chinese Academy of Sciences\\Beijing, China, 100190}
\email{jzhang87@amss.ac.cn}

\address{\dag\ Departament de Matem\`atiques, Universitat Polit\`ecnica de Catalunya}
\email{tere.m-seara@upc.edu}
\thanks{}
\subjclass[2010]{Primary 37N05, 37D10; Secondary 70F07, 70H09}
\keywords{}
\date{}
\begin{document}
\maketitle

\centerline{\scshape Tere M. Seara$^{\dag}$,\quad \scshape Jianlu Zhang$^*$}

\begin{abstract}
The restricted planar four body problem describes the motion of a massless body
under the Newtonian gravitational force of other three bodies (the primaries), of which the motion gives us general solutions of the three body problem.

A trajectory is called {\it oscillatory} if it goes arbitrarily faraway but returns infinitely
many times to the same bounded region. We prove the existence of such type of trajectories provided the primaries evolve in suitable periodic orbits. 
\end{abstract}
\vspace{20pt}

\tableofcontents

\section{Introduction}\label{s1}
\vspace{20pt}

The {\it Restricted Planar Four Body Problem} (RP4BP from now on) models the motion of
a body of zero mass under the Newtonian gravitational force of three other bodies (the {\it primaries}),
which evolve in general planar  three body motion. 
Usually the RP4BP can be interpreted as a {\it Sun-Jupiter-Planet-Asteroid} (S-J-P-A) system. 
We can normalize the mass of the Sun and Jupiter by $1-\mu$ and $\mu$ individually, with $\mu\in(0,1/2]$. 
For us $\mu$ is a fixed positive parameter, so in the following paragraph we do not write this dependence explicitly. 
When the mass of the Planet, denoted by $\dt$, is suitably small, 
we will  find certain  periodic orbits of the S-J-P subsystem (see Theorem \ref{rot-continue}). 
In Cartesian coordinates, if we denote the position of the primaries by $x_{S}$, $x_{J}$ and $x_{P}\in\R^{2}$ respectively, the periodic orbit will satisfy
\[
\big(x_{S}(t+T_{\dt}),x_{J}(t+T_{\dt}),x_{P}(t+T_{\dt})\big)=\big(x_{S}(t),x_{J}(t),x_{P}(t)\big),\quad\forall\ t\in\R
\]
with $T_{\dt}\in\R_{+}$ 
being a constant continuously depending on, and uniformly bounded as $\dt \to 0^+$. 

%

Moreover, the periodic orbit is given, in first order in $\dt$, by circular orbits  of the  Sun $x_S$ and Jupiter $x_J$ of radious $\mu$ and $1-\mu$ respectively, and a {nearly circular orbit for the planet $x_P$ of size ${1}/{\eps_\dt^2}$, where $\eps_\dt$ is a small parameter uniformly bounded as $\dt\to 0^+$, so that the planet is far away from the Sun and Jupiter.  See Figure  \ref{fig-4bp}.

\begin{figure}
\begin{center}
\includegraphics[width=12cm]{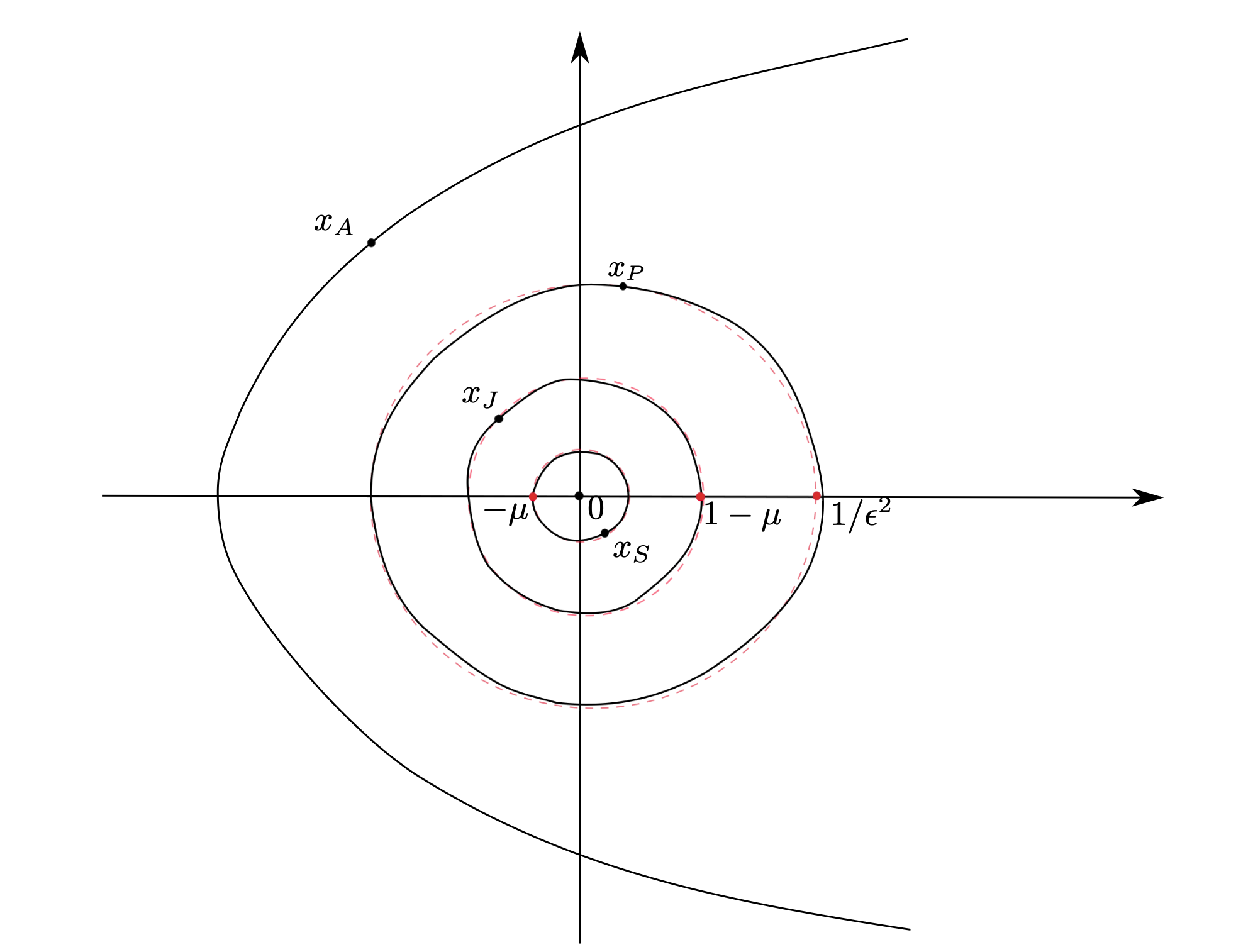}
\caption{The red dash circles describe the position of S-J-P as $\dt=0$. Take $0<\eps_{\dt}\ll1$ and make $\dt\ll\eps_{\dt}$ smaller, the comet type periodic motion should be a deformation of the red circles. The massless Asteroid performs a parabolic motion, of which dist$(x_A,0)\gtrsim \cO(\eps_{\dt}^{-2})$ is needed, to avoid a collision between A and P.}
\label{fig-4bp}
\end{center}
\end{figure}

%
%
When the primaries move in this periodic orbit,   the motion of the asteroid can be described by the following Hamiltonian:
\be\label{ham-4bp}
H_A(x_A,y_A, t)=\frac12|y_A|^2-V_A(x_{A}, t)
\ee
where $x_{A},\ y_{A}\in\R^{2}$ and
\[
V_A(x_{A}, t)=\frac{1-\mu}{|x_A-x_S(t)|}+\frac{\mu}{|x_A-x_J(t)|}+\frac{\dt}{|x_A-x_P(t)|}.
\]

The purpose of this paper is to show
the existence of some particular orbits of this Hamiltonian system: the {\it oscillatory orbits}. 
This kind of orbits can leave every bounded region but return infinitely many times to the same bounded region. 
If $\dt=0$, the system becomes the {\it Restricted Planar Circular Three Body Problem} (RPC3BP), of which the oscillatory orbits has been found in \cite{GMS}. 
Nevertheless, contrarily to what happens in Arnold diffusion, as oscillatory orbits require infinite time, their existence can not be obtained just using the regular dependence on parameters of the Hamiltonian.
The mechanism used in this paper follows the lines of \cite{GMSS}. The 
main idea to obtain these orbits is to study the so-called manifold of infinity, which, in suitable coordinates,  turns out to be an invariant manifold with stable and unstable manifolds which intersect transversally. 
This intersection will allow us to define some  {\it scattering map} and study the associated recurrence of trajectories. 
The whole approach is well developed in a series of papers, \cite{DLS1, DLS, DLS2}, in the study of Arnold diffusion in nearly integrable Hamiltonian systems. 

In the current paper for fixed $\mu\in(0,1/2]$ and 
$0<\dt\ll
1$
sufficiently small, we combine the acquisition of periodic orbits for the 3BP with the previously introduced scattering map for the RP4BP, and obtain the oscillatory orbits for system (\ref{ham-4bp}). 
On both parts we work in a nearly  integrable setting and use perturbative methods. 
On one side, we can get the  desired periodic orbits for the 3BP, as a continuation of certain periodic orbits from the RPC3BP if $\dt$ is small enough. 
Since there is no restriction on the $\mu$ value, the periodic orbits we found are always of {\it comet-type}, i.e. the relative distance between the Planet and the Sun-Jupiter couple is large.
On the other side, taking the  previous periodic orbits into (\ref{ham-4bp}) we get a system which is  a $\mathcal{O}(\delta)$ time-periodic  perturbation of the RPC3BP system, where the transversality of the stable and unstable manifolds of the infinity manifold has been proved in \cite{GMS}. 
This allows us to compute the perturbed scattering map which will be nearly integrable.
Therefore, we can apply the twist  theorem   to find certain invariant sets acting as a skeleton that oscillatory orbits will follow.\\

To prove  the existence of ``comet-type'' periodic orbits (named by \cite{MHO}) for the 3BP rigorously and obtain quantitative estimates for them, we use a a matured continuation method inherited from the RPC3BP.
Although other types of periodic orbits have been already found, e.g. the famous Figure-8 orbits \cite{CM}, technically that demands a equi-mass setting which can not be guaranteed in our case. Moreover,  the obtained  periodic orbits for the 3BP have a ``natural limit''  for $\dt\to 0^+$,  and this makes system (\ref{ham-4bp}) to be a $\cO(\dt)$- pertubation of the RPC3BP.
\\

Another fact we want to claim is the continuation method from the RPC3BP ($\dt=0$) to the 3BP ($\dt >0$) is rather {\it robust}. Besides the comet-type periodic orbits, we can also find the {\it second type} elliptic  periodic orbits, or quasi-periodic orbits with irrational frequencies. These orbits will give us totally different RP4BP systems, of which the oscillatory orbits could still be found, by more complicated analysis. All the evidence shows the abundance of the oscillatory orbits in the phase space. Moreover, extract new mechanisms of such orbits from these systems would be rather meaningful to this topic.

\subsection{The abundance of the oscillatory orbits in the 3BP} \label{s1s1}For the 3BP (either restricted
or non restricted, planar or spatial), singular solutions which correspond to the collision exist for finitely long time. As early as Siegel's times \cite{Si}, people surmise that the {\it collision orbits} should be dense in suitably region of the phase space. This is the well known {\it Siegel's Conjecture} and was formalized by Alexseev in 1970's \cite{A}. In a recent work \cite{GKZ}, we gave an estimate of the asymptotic density of the collision orbits for the RPC3BP, which indicated the collision orbits should be numerically dense in the phase space.\\

Beyond the collision orbits, all the other solutions of the 3BP are well defined for $t\in\R$. So an important question is to study the final motions of these regular orbits. This work was initiated by Chazy in 1922 \cite{Cha}, when he gave a complete classification of the possible final motions (see \cite{A} for more details). Of all his classifications, the oscillatory motion is definitely the most erratic type, which can be formalized by the following:
\[
\cO\cS^\pm\ \text{(oscillatory): }\limsup_{t\rightarrow\pm\infty}\|x\|=+\infty,\quad\liminf_{t\rightarrow\pm\infty}\|x\|<+\infty.
\]
Only until 1960 this kind of motion was firstly discovered by Sitnikov \cite{S}, in a restricted spatial model. After that, Moser gave a different proof for the  Sitnikov's model which strongly influenced the subsequent results in the area (see \cite{Mo}). 
Following Moser's idea, the works \cite{GMS,LS,LS2,Moeck,X} obtained  oscillatory motions in other generalized settings.
\\

Thanks to all these efforts, now we have a comparably clear understanding on the mechanisms of the oscillatory motion, but it's still too faraway to figure out the portion of this kind of orbits in the whole phase space. The famous {\it Kolmogrov's Conjectured} guesses that {the Lebesgue measure of the set of the oscillatory orbits should be zero} \cite{A}. Nonetheless, there is evidence in the recent work \cite{GK}, which showed that the Hausdorff dimension of the set of oscillatory motions for the Sitnikov
example (and the RPC3BP) could reach maximal for a Baire's generic subset of an open set of parameters (the eccentricity of the primaries in the Sitnikov's example and the mass ratio in the RPC3BP).\\

\subsection{Arnold diffusion in the N Body Problem ($N\geq3$)}\label{s1s2}
 For  a  nearly integrable system in action-angle coordinates
\[
H(\phi,I)=h(I)+\eps H_1(\phi, I), \quad \phi\in\T^n,\ I\in\cB\subset\R^n, \ \eps\ll1,
\]
the {\it Arnold diffusion problem } analyzes the drastic
changes that the action variables $I$ can undergo.
 Due to the restriction of the dimension and the existence of KAM tori, this kind of phenomenon can only be found for $n\geq 3$. 
Recent works, \cite{BKZ, CY1, CY2, C1,C2, DLS1, DLS, GM,KZ,Tr} among them, have proved the existence of Arnold diffusion  for typical nearly integrable systems, by using geometric and variational methods.
 
For the  NBP, one can expect to prove Arnold diffusion in certain regions of the phase space, once the nearly integrable structure  is established. One quantity that can be studied in several cases is the angular momentum of the diffusion orbits, to see that it  should make big changes in a rather long time. 
As far as we know, the first paper dealing with Arnold diffusion in Celestial Mechanics is \cite{Moe}, where the author concerned a five body model. 
In \cite{DGR}, the authors analyze unstable behavior for the three body problem close to the Lagrangian point $L_1$. 
In the recent work \cite{DKRS}, the authors proved the existence of Arnold diffusion for the {\it Restricted Planar Elliptic Three Body Problem} (RPE3BP) with exponentially small mass $\mu$ of the Jupiter. Let us stress here that the RPE3BP has the minimum required dimension of all the models permitting Arnold diffusion.

For the  RP4BP, \cite{Xu} showed a mechanism for the existence of diffusion orbits, but the proof assumed the transversality of  certain invariant manifolds which has been only checked numerically.
More numercial and analytical evidence on Arnold diffusion of RPE3BP and RSC3BP can be found in \cite{CGL,DGR2,FGKR,X2} and some  quantitative estimates  of  Arnold Diffusion and stochastic behavior in the Three-Body Problem is given in \cite{CG}.
Even if in this paper we do not deal with diffusion orbits, both the existence of diffusion or the existence of oscillatory orbits, share a common setting, to establish the transversal intersection of some stable and unstable manifolds of a normally hyperbolic (or parabolic) invariant manifold and then study the associated scattering map. For this reason we think that in the present example one can try to proof the existence of diffusing orbits in a future work. See Remark \ref{diff-rmk}.


\subsection{Main result}\label{s1s3} 
Now we obtain our main result as the following:

\begin{thm}\label{main-thm}
Fix any value of $\mu\in(0,1/2]$. Then,  there exists $0<\dt_0 \ll1$, such that for any $\dt\in[0,\dt_0]$ we have:
\begin{itemize}
\item
The 3BP of S-J-P has  a periodic orbit
$(x_S (t), x_J(t), x_P(t))=(x_S^{\dt} (t), x_J^{\dt}(t), x_P^{\dt} (t)) $ of period $T_{\dt}$:
\[
T_{\dt}=2\pi q\Big(1+\cO(\sqrt\dt)\Big)
\]
and $q$ is an integer independent of $\dt$.
\item
The RP4BP given by the Hamiltonian system of Hamiltonian (\ref{ham-4bp}), has  forward oscillatory orbits $(x_A(t),y_A(t))$. 
Namely, they satisfy:
\[
\limsup_{t\rightarrow+\infty}\|x_A(t)\|=+\infty,\quad\liminf_{t\rightarrow+\infty}\|x_A(t)\|<+\infty.
\]
\end{itemize}
\end{thm}

As happens in \cite{GMSS} the same  mechanism can also be used to  construct backward oscillatory orbits (for $t\rightarrow-\infty$)
but not  to show the existence of bilateral oscillatory orbits. To get these orbits requires the construction of a horseshoe and this is beyond the goals of this paper.
Notice that $\dt=0$ is allowable and (\ref{ham-4bp}) will degenerate to the RPC3BP, on which the forward and backward oscillatory orbits have been found in \cite{GMS}. That's why $\delta=0$ is included in our result. \\

Let us stress here that to show the existence of ``comet-type'' periodic orbits or quasi periodic orbits for the general 3BP ($\dt\lesssim \cO(\mu)$) is still unknown. 
Current techniques highly rely on the nearly integrable structure. 
This is one of the reasons  why we need $\dt_0$ be sufficiently small. 
The condition $\mu>0$ is also natural and
without loss of generality, $\dt_0<\mu$ can be  assumed since we are working in  a perturbed setting 
$\dt\to 0^+$.
%
If so, $\mu\rightarrow 0$ will compel $\dt_0\rightarrow 0$ and (\ref{ham-4bp}) degenerate to Two Body Problem, which is naturally integrable. 
The oscillatory motion couldn't happen for this case.

\begin{rmk}\label{diff-rmk}
Our system (\ref{ham-4bp}) shares with the RPE3BP that it is a $\cO(\dt)$-periodic perturbation of the RPC3BP and therefore is a two and a half degrees of freedom system. 
In the work \cite{DKRS}, the authors showed that, after checking  some nondegeneracy conditions for the scattering map (given in Section \ref{s3}), they could obtain diffusion orbits.
Precisely, there are two different scattering maps associated with two different homoclinic channels of the manifold of infinity,  
each of which is an area preserving twist map on a cylinder. 
The nondegeneracy claims that these two scattering maps do not have common invariant curves. Then,  combining the two scattering maps, in \cite{DKRS} orbits with a large drift in the angular momentum where obtained. 

We think that these ideas can be also used in the RP4BP, if the mentioned non-degeneracy condition can be checked. But this requires some non-trivial computations and we leave it for future work. 
\end{rmk}

\subsection{Scheme of the proof}\label{s1s4}

In this section we will give a scheme which applies for both Theorem \ref{main-thm} and Remark \ref{diff-rmk}.  More detail will be supplied in Section \ref{s4}.\\
%
%
Let's first review the idea of constructing oscillatory orbits for the RPC3BP in \cite{LS, GMS}.
As the RPC3BP has a first integral, the {\it Jacobi Constant}, when written in rotationg coordinates it becomes an autonomous Hamiltonian System of two degrees of freedom. Fixing the energy level (that at infinity coincides with the angular momentum) and taking a global surface of section it can be
reduced to a two dimensional Poincar\'e map of which the `infinity'  $\{|q|=+\infty,\  \dot q = 0\}$ is a parabolic fixed point. 

Just like in the hyperbolic case, it inherits stable (resp. unstable) invariant manifolds which intersect transversely as proved in \cite{GMS} for any value of the mass parameter $0<\mu\le \frac12$. 
This intersection gives rise to some  symbolic dynamics  as Moser proved in \cite{Mo} for the Sitnikov problem and  Sim\'{o} and Llibre in \cite{LS} for the RPC3BP, which supplies us orbits traveling close to the invariant manifolds and the $\liminf$ of the distances to the fixed point, which corresponds to the infinity in the original coordinates,  is zero.
\\

Notice there are two crucial ingredients in previous strategy: the transversality of the parabolic invariant skeleton and the symbolic dynamics. To apply this strategy to system (\ref{ham-4bp}), we have to achieve both two or find reasonable substitutes.
\\

(I). 
For system (\ref{ham-4bp}) the phase space is  of dimension five. 
Therefore the associated Poincar\'e map becomes four dimensional and infinity becomes a two dimensional cylinder with one angular variable and an ``action variable'', the angular momentum of the mass-less body (see Section \ref{s3}). 
Although this cylinder is still {\it normally parabolic} and has invariant manifolds, 
we have to additionally show this cylinder is {\it homogenous}, i.e. it consists of fixed points. 
This is done in Theorem \ref{rot-continue}, by using a continuation approach with $\dt\ll1$. 
Besides, as a perturbation of the RPC3BP, if we remain in a compact subset, the invariant manifolds of this cylinder still intersect transversaly for $\dt\ll1$.
\\

(II). 
Because of the increase of dimension we 
use a method borrowed from the construction of transition chains of the Arnold diffusion problem \cite{Ar} and proposed in \cite{GMSS} to obtain the oscillatory orbits. 
Precisely, we find a sequence of fixed points belonging to a compact region of the cylinder of infinity which are connected by  heteroclinic orbits. 
These orbits  form a so called {\it  infinite transition chain}, and, if we successfully obtain an orbit shadowing the whole chain, then we get an oscillatory orbit.
To find the transition chain, we get a nearly integrable scattering map 
in subsection \ref{s3s2} and apply the KAM theorem to it. Any KAM torus will supply the uniform compactness, so we just need to choose the sequence on the torus.
\\


Recall that the vertical direction of the cylinder of infinity can be parameterized by the angular momentum of the orbits. 
So another inspiring question is to find a suitable transition chain of periodic orbits  on the cylinder of infinity with large change of the angular momentum. 
Shadowing this chain the diffusion orbits can be constructed (see Remark \ref{diff-rmk}). 
It seems to be a totally opposite question to the construction of the oscillatory orbits, and strongly relies on the dynamics of the scattering map. Essentially, as in our problem the scattering map has invariant  KAM tori, these tori are an  obstruction 
to obtain orbits with big increase of the angular momentum by only one scattering map.
We have to use two scatering maps to build a sequence of points which breaks the obstruction of the KAM tori and makes persistently upward (resp. downward) movement \cite{DLS,DLS2,DKRS}. 
To check this mechanism  also requires further quantitative analysis and necessary refinement of the model (\ref{ham-4bp}) in our future works.\\

\subsection{Organization of the article}\label{s1s5} 

The paper is organized as follows. First in Section \ref{s2} we prove the first item of Theorem \ref{main-thm}: we prove the existence of periodic orbits for the 3BP, and show that they  are  continuation from the  ones of the RPC3BP. 
In Section \ref{s3} we prove the second  item of Theorem \ref{main-thm}: we consider the three  primaries moving in the obtained periodic orbit to get the designated system (\ref{ham-4bp}) for the RP4BP and  prove the existence of oscillatory orbits for this system.
First in section  \ref{setup} we write the Hamiltonian giving the RP4BP in suitable coordinates and  analyze the existence and transversal intersection of the  stable and unstable invariant manifolds of the ``manifold of infinity''.
Section \ref{s3s1} is devoted to  recall the known facts for the case $\dt=0$, which becomes the $RPC3BP$. As in this case the needed transversality properties are known, classical perturbation theory allows us to construct the needed transition chain of periodic orbits through the study of the scattering map in section \ref{s3s2}. Finally, in Section \ref{s4}, we state the shadowing mechanism which gives  the oscillatory orbits, which technically relies on a $\lambda-$lemma applied to the invariant manifolds of the normally parabolic cylinder of infinity given in \cite{GMSS}. 
For readability we moved parts of  some  coordinate transformations to the Appendix.

\vspace{20pt}

\noindent{\bf Acknowledgement.} T.S. was partially supported by the MINECO-FEDER Grant  PGC2018-098676-B-100 (AEI/FEDER/UE), the Catalan Grant 2017SGR1049, and the Catalan Institution for Research and Advanced Studies via an ICREA Academia Prize 2019.  
J.Z.  is supported by the National Natural Science Foundation of China (Grant No. 11901560). 
This material is based upon work supported by the National Science Foundation under Grant No. DMS-1440140 while the authors were in residence at the Mathematical Sciences Research Institute in Berkeley, California, during the Fall 2018 semester.


%
%

\section{Periodic orbits for the 3BP}\label{s2}
\vspace{20pt}
%
%
In this section  we prove the first item of Theorem \ref{main-thm}: we will see how to find some periodic solutions for the S-J-P model. 
Basically these periodic solutions can be considered as the continuation from the RPC3BP to the 3BP system. 
As far as we know, \cite{H} first proposed a suitable coordinate of which the 3BP can be translated into a Lagrangian variational problem with three degrees of freedom. To get our desired periodic orbits, we adapt the language of \cite{H} to the Hamiltonian setting.

\subsection{Symplectic transformations for 3BP}\label{s2s1}

Let's start with the following 
$6-$degrees of freedom Hamiltonian system 
\be\label{3bp-cart}
\mathcal{H}(x_P,x_J,x_S,y_P,y_J,y_S)&=&\sum_{i=P,J,S}\frac{|y_i|^2}{2m_i}-\sum_{i\neq j}\frac{\cG m_im_j}{|x_i-x_j|}\nonumber\\
&=&\frac1{2\dt}|y_P|^2-\dt\Big[\frac{1-\mu}{|x_P-x_S|}+\frac{\mu}{|x_P-x_J|}\Big]+\nonumber\\
& &\Big[\frac1{2(1-\mu)}|y_S|^2+\frac1{2\mu}|y_J|^2-\frac{\mu(1-\mu)}{|x_S-x_J|}\Big]
\ee 
where we take $\cG=1$, $m_S=1-\mu$, $m_J=\mu$ and $m_P=\dt$. 
Recall that there exists a bunch of first integrals we can use, i.e. 
\be\label{first-int}
\left\{
\begin{aligned}
(1-\mu)x_S+\mu x_J+\dt x_P=0,\\
y_S+ y_J+ y_P=0.
\end{aligned}
\right.
\ee 
If we transfer (\ref{3bp-cart}) to the Jacobi coordinates by the following
   \be\label{jacobi-tran}
   \Phi:\left\{
   \begin{aligned}
   Q_0&=x_S,& \\
   Q_1&=x_J-x_S,& \\
   Q_2&=x_P-(1-\mu)x_S-\mu x_J,& \\
   P_0&=y_S+ y_J+y_P,&\\
   P_1&= y_J+\mu y_{P},& \\
   P_2&=y_P,& 
   \end{aligned}
   \right.
   \ee
   the Hamiltonian becomes independent of $Q_0$, therefore $P_0$ is a first integral.
   From now on we choose $P_0=0$ as (\ref{first-int}) shows and we obtain the $4$-degrees of freedom Hamiltonian:
   \be
    H(Q_1,P_1, Q_2,P_2)&=&\frac{\dt+1}{2\dt}|P_2|^2-\dt\Big[\frac{1-\mu}{|Q_2+\mu Q_1|}+\frac{\mu}{|Q_2-(1-\mu)Q_1|}\Big]\nonumber\\
     & &+\Big[\frac{|P_1|^2}{2\alpha}-\frac{\alpha}{|Q_1|}\Big]
      \ee
      with $\alpha:={\mu(1-\mu)}$. 
      For convenience, we can further write $Q_1$ in polar coordinates, namely there exists a symplectic transformation 
      \[\Phi_{pol}:
\left\{
\begin{split}
\pi_1Q_1&=r\cos\theta, \\
\pi_2Q_1&=r\sin\theta,\\
\pi_1P_1&=R\cos\theta-\frac{\Theta}{r}\sin\theta, \\
\pi_2P_1&=R\sin\theta+\frac{\Theta}{r}\cos\theta,
\end{split}
\right.
\]
 such that 
\[
dQ_1\wedge dP_1+dQ_2\wedge dP_2=dr\wedge dR+d\theta\wedge d\Theta+dQ_2\wedge dP_2,
\]
 of which the Hamiltonian becomes
\be
     H^*(r,R,\theta,\Theta,Q_2,P_2)&=&\frac{\dt+1}{2\dt}|P_2|^2-\dt\Big[\frac{1-\mu}{|Q_2+\mu (r\cos\theta,r\sin\theta)|}+\frac{\mu}{|Q_2-(1-\mu)(r\cos\theta,r\sin\theta)|}\Big]\nonumber\\
     & &+\Big[\frac{1}{2\alpha}(R^2+\frac{\Theta^2}{r^2})-\frac{\alpha}{r}\Big].\nonumber
\ee
Then we further take the following {\it Hadjidemetriou's rotating} coordinates
\be\label{rotat-3bp-tran}
(r,R,\theta,\Theta,Q_2,P_2)\xrightarrow{\Phi_{had}}(r,R,\theta,\Om,q_2,p_2)
\ee
with
      \[
     (q_2,p_2)=(e^{-i\theta}Q_2,e^{-i\theta}P_2),\quad e^{-i\theta}=\begin{pmatrix}
    \cos\theta  &\sin\theta    \\
    -\sin\theta  &  \cos\theta
\end{pmatrix}
      \]
This transformation $\Phi_{had}$ is symplectic: 
\be
dr\wedge dR+d\theta\wedge d\Theta+dQ_2\wedge dP_2&=&
dr\wedge dR+d\theta\wedge d\Om+dq_2\wedge dp_2.
\ee
and  we call
      \be
  \Om&=&x_P\times y_P+x_S\times y_S+x_J\times y_J\nonumber\\
  &=&Q_1\times P_1+Q_2\times P_2\nonumber\\
  &=&\Theta+Q_2\times P_2.
  \ee
to the {\it total angular momentum}. 

So we finally get an operable Hamiltonian
      \be\label{polar sys-4bp}
      H_{rot}^\dt(r,R, \theta,\Om,q_2,p_2)&=&\frac{\dt+1}{2\dt}|p_2|^2-\dt\Big[\frac{1-\mu}{|q_2+r(\mu ,0)|}+\frac{\mu}{|q_2-r(1-\mu,0)|}\Big]\nonumber \\
      &&+\Big[\frac{R^2}{2\alpha}+\frac1{2\alpha r^2}(\Om-q_2\times p_2)^2-\frac{\alpha}{r}\Big].
      \ee
 As $H_{rot}^\delta$ does not depend on $\theta$,  $\Om$ is an first integral and
 we can restrict 
  \be
  \Om=\alpha=\mu(1-\mu).
  \ee
Now the system $H_{rot}^\dt$ is of 3 degrees of freedom.
Abusing notation we write $H_{rot}^\dt (r,R,q_2,p_2)$.

   \begin{rmk}
   Previous transformations are all explicit, therefore, once we find a periodic orbit of system $H_{rot}^\dt$, we can instantly pull it back to obtain its position in Cartesian coordinates by the following:
   \be\label{exp-comb-sym-tran}
     \left\{
   \begin{aligned}
   x_S&=-\mu r\begin{pmatrix}
       \cos\theta  \\
      \sin\theta  
\end{pmatrix}-\frac{\dt}{1+\dt}\begin{pmatrix}
    \cos\theta  &-\sin\theta    \\
    \sin\theta  &  \cos\theta
\end{pmatrix}q_2,& \\
   x_J&=(1-\mu) r\begin{pmatrix}
       \cos\theta  \\
      \sin\theta  
\end{pmatrix}-\frac{\dt}{1+\dt}\begin{pmatrix}
    \cos\theta  &-\sin\theta    \\
    \sin\theta  &  \cos\theta
\end{pmatrix}q_2,& \\
   x_P&=\frac{1}{1+\dt}\begin{pmatrix}
    \cos\theta  &-\sin\theta    \\
    \sin\theta  &  \cos\theta
\end{pmatrix}q_2.& 
   \end{aligned}
   \right.
   \ee
   Now we claim the existence of periodic orbits in the following statement:
   \end{rmk}
   \begin{thm}\label{rot-continue}
   Fix any  $0\leq\mu\leq1/2$. 
   There exist $0<\eps_0=\eps_0(\mu)\ll1$ and $\dt_0=\dt_0(\eps_0)>0$, such that for any 
   $0<\dt\leq\dt_0$, the system (\ref{polar sys-4bp}) has two periodic orbits
   $\gamma_{\dt}^\pm$ of period $T_{\dt}^\pm$, which can be expressed by 
  \[
  \gamma_{\dt}^\pm(t):=\Big\{\Big(r^\pm(t),\theta^\pm(t),R^\pm(t),\Om(t),q_2^\pm(t),p_2^\pm(t)\Big)\Big\}\subset \R\times\T\times\R^2\times\R^4
  \]
  and
  \[
  \gamma_{\dt}^\pm(t+T_{\dt}^\pm)=\gamma_{\dt}^\pm(t),\quad\forall t\in\R.
  \]
   More precisely, for any integer $q\in[\frac43\eps_0^{-3},\frac83\eps_0^{-3}]$ fixed, 
   there always exists $\eps_{\dt}^\pm\in(\eps_0/2,\eps_0)$ such that 
   \be
T_{\dt}^\pm=2\pi q\Big(1+\cO(\sqrt{\frac{\dt}{{\eps_{\dt}^\pm}}})\Big)
   \ee
   and the following estimate holds:
   \be\label{rot-conti-est}
 \left\{
   \begin{aligned}
   \|r^\pm(t)-1\|&\leq \cO(\frac{\dt}{\eps_{\dt}^{\pm4}}),& \\
   \|\dot\theta^\pm(t)-1\|&\leq \cO(\frac{\dt}{\eps_{\dt}^{\pm4}}),&\\
   \|R^\pm(t)\|&\leq \cO(\frac{\dt}{\eps_{\dt}^{\pm4}}),&\\
   \Om(t)&= \alpha.& \\
     \Bigg\|q_2^\pm(t)-\frac1{\eps_{\dt}^{\pm2}}\begin{pmatrix}
      \cos\frac{2\pi t}{T_{\dt}^\pm}\\[4pt]
      \sin\frac{2\pi t}{T_{\dt}^\pm}  
\end{pmatrix}\Bigg\|&\leq \cO(\eps_{\dt}^{\pm2}),&\\
   \Bigg\|p_2^\pm (t)\pm\dt\eps_{\dt}^\pm\begin{pmatrix}
      \sin\frac{2\pi t}{T_{\dt}^\pm}    \\[4pt]
      -\cos  \frac{2\pi t}{T_{\dt}^\pm}
\end{pmatrix}\Bigg\|&\leq \cO(\dt\eps_{\dt}^{\pm5}),& \\
   \end{aligned}
   \right.
\ee
\end{thm}

\begin{rmk}\label{rmk:uni}
During the proof of Theorem \ref{rot-continue}, we can see that, as $\dt \to 0^+$, 
the periodic orbits $\gamma_\dt^\pm$ tend to certain periodic orbits $\gamma_*^\pm$ of the RPC3BP with the period $T_*^\pm=2\pi q$ being the limit of $T_\dt^\pm$. 
Besides, formally we have 
\[
\gamma_*^\pm(t)=\Big\{\Big(r_*^\pm(t),\theta_*^\pm(t),R_*^\pm(t),\Om_*(t),q_{2,*}^\pm(t),p_{2,*}^\pm(t)\Big)\Big\}\subset \R\times\T\times\R^2\times\R^4
\]
with 
\[
r_*^\pm (t)\equiv1, \ \theta_*^\pm (t)=t, \  R_*^\pm (t)\equiv0,\ \Om_*(t)\equiv\mu(1-\mu),
\]
and
\[
 q_{2,*}^\pm(t)=-\frac1{\eps_*^{\pm2}}\begin{pmatrix}
      \cos \frac tq \\[4pt]
      \sin \frac tq 
\end{pmatrix}+ \cO(\eps_{*}^{\pm2}), \  
v_{2,*}^\pm(t):=\lim_{\dt\rightarrow 0^+}\frac{p_{2}^\pm (t)}{\dt}=\mp\eps_*^\pm\begin{pmatrix}
      \sin \frac tq    \\[4pt]
      -\cos  \frac tq
\end{pmatrix}
\]
for certain $\eps_*^\pm$ being the limit of $\eps_\dt^\pm$. 
Although system \eqref{polar sys-4bp} has a singular limit as $\dt\rightarrow 0^+$, by a suitable rescaling transformation we get   a system $\wt H^\dt_{res}$ in (\ref{eq:Hres}) which indeed has a regular limit as $\dt\rightarrow 0^+$, namely the RPC3BP.
\end{rmk}


\begin{proof}
To proof this theorem we need to perform several changes of variables.
 In the first part of the proof, we consider the Hamiltonian system of $H_{rot}^\dt(r, R, q_2, p_2)$ as a system of three degrees of freedom, that is, we work in the variables $(r, R, q_2, p_2)$ and take the parameter $\Om= \alpha=\mu(1-\mu)$. 
 Then, the theorem will be a straightforward application of Proposition  \ref{prop-conti-peri}, once the Hamiltonian system of $H_{rot}^\dt$ is written in the suitable coordinates.

Now we describe the changes we perform, the details are given in Appendix \ref{A}. 
First we need to transfer the Hamiltonian system of $H_{rot}^\dt$ in (\ref{polar sys-4bp}), wich depends singularly on $\dt$, into a regular perturbation of  the RPC3BP. 
This can be achieved with a rescaling transformation $\Phi_{res}^1$ in \ref{trans-res}, namely we take
\[
(r,R,\theta,q_2,p_2) \xrightarrow{\Phi_{res}^1}(\wt r,\wt R,q_2,v_2)
\]
by
\[
\left\{
   \begin{aligned}
   p_2&=\dt v_2,&\\
   r&=1+\sqrt\dt\wt r,& \\
   R&=\sqrt\dt\wt R,& 
   \end{aligned}
   \right.
\]
which transforms $H_{rot}^\dt$ into (see  
\eqref{rescal-sys-4bp}, \eqref{eq:Hres},\eqref{eq:Hresdelta}, \eqref{f-fun}, \eqref{g-fun})
\be
 \wt H_{res}^\dt(\wt r,\wt R, q_2,v_2)&=& 
\underbrace{\Big[  -q_2\times v_2 +  \frac12|v_2|^2-\frac{1-\mu}{|q_2+ (\mu,0)|}-\frac{\mu}{|q_2-(1-\mu,0)|} \Big]}_{RPC3BP}\nonumber\\
  & &+\underbrace{\frac12\Big[\frac{\wt R^2}{\alpha}+\alpha\wt r^2\Big]}_{rotator}+\underbrace{\Delta \wt H_{res}^\dt(\wt r,q_2, v_2)}_{reminder}.\nonumber
\ee
Next, we can constraint $\wt H_{res}^\dt$ to certain domain of the phase space, where we expect to find the comet-type periodic orbits.
 For this purpose we apply another rescaling transformation $\Phi_{res}^2$ in (\ref{trans-res-2}):
\[
(\wt r,\wt R,q_2,v_2) \xrightarrow{ \Phi_{res}^2}(\wh r,\wh R,\wh q_2, \wh v_2)
\]
with
\[
\ q_2=\frac{\wh q_2}{\eps^2},\quad  v_2=\eps \wh v_2,\quad \wt r=\frac{\wh r}{\sqrt\eps},\quad\wt R=\frac{\wh R}{\sqrt\eps}
\]
where $0<\eps\ll1$ is a small parameter that will be fixed later on. 
The new system becomes (see \eqref{rescal-sys-4bp-2}), \eqref{f-fun1},\eqref{g-fun1}), 
\be
\wh H_{res}^{\dt, \eps}(\wh r, \wh R,\wh q_2,\wh v_2) &=&\Big[-\wh q_2\times \wh v_2+\eps^3(\frac{|\wh v_2|^2}2-\frac{1}{|\wh q_2|})+\cO(\frac{\mu\eps^7}{|\wh q_2|^3})\Big]\nonumber\\
& &+\frac12\Big[\frac{\wh R^2}{\alpha}+\alpha\wh r^2\Big]+
 \Dt \wh H_{res}^{\dt, \eps}(\wh r, \wh q_2,\wh v_2).\nonumber
\ee
%
%


Finally, we write the second body part $(\wh q_2,\wh v_2)$ of previous variables in symplectic polar coordinates:
\[
(\wh q_2,\wh v_2)\xrightarrow{\Phi_{pol}} (\rho,\Upsilon,\phi,G), \quad\text{via } \left\{
\begin{split}
\pi_1\wh q_2&=\rho\cos\phi, \\
\pi_2\wh q_2&=\rho\sin\phi,\\
\pi_1\wh v_2&=\Upsilon\cos\phi-\frac{G}{\rho}\sin\phi, \\
\pi_2\wh v_2&=\Upsilon\sin\phi+\frac{G}{\rho}\cos\phi,
\end{split}
\right.
\]
then we get the final Hamiltonian (see (\ref{res-sys-comet}), \eqref{exp-R}), \eqref{exp-f},\eqref{exp-g}):
\be\label{eq:ham-res-wh}
H_{res}^{\dt, \eps}(\wh r,\wh R,\rho,\Upsilon,\phi,G)&=&-G+\eps^3\Big[\frac12(\Upsilon^2+\frac{G^2}{\rho^2})-\frac1\rho\Big]
+
\frac12\Big[\frac{\wh R^2}{\alpha}+\alpha\wh r^2\Big]\nonumber\\
& &+\cO(\frac{\mu\eps^7}{\rho ^3})
+\Dt H_{res}^{\dt, \eps}(\wh r,\rho,\phi,\Upsilon,G).
\ee
Now, for any given $C>1$,  we consider the  bounded domain:
\begin{equation}\label{eq:Dres}
(\wh r,\wh R,\rho,\Upsilon,\phi,G)\in \cD_{res}:=B_{\R^4}(0,C)\times\T\times[-C,C]\subset\R^4\times\T\times\R ,
\end{equation} 
%
%
and  we apply  Proposition \ref{prop-conti-peri} to the Hamiltonian system of $H_{res}^{\dt, \eps}$ in \eqref{eq:ham-res-wh}: there exist $\eps_1>0$ and $\dt_1 >0$ such that for any $0<\eps\leq\eps_1$ and $0\leq \dt\leq\dt_1\eps^7$, the system 
$
H_{res}^{\dt, \eps}$ has two periodic orbits
\[
 \wh\gamma_{\dt,\eps}^\pm(t)=\Big\{\Big(\wh r^\pm(t),\wh R^\pm(t),\rho^\pm(t),\Upsilon^\pm(t),\phi^\pm(t),G^\pm(t)\Big)\Big|   \  t\in[0,\wh T_{\dt,\eps}^\pm]\Big\}
\]
satisfying \eqref{est-peri-res} and 
contained in the domain $\cD_{res}$ with the period $\wh T_{\dt,\eps}^\pm$ (given in (\ref{peri-res-orbit}))
\[
 \wh T_{\dt,\eps}^\pm=\frac{2\pi}{1\mp\eps^3+\cO(\eps^7)+\cO(\sqrt{\dt/\eps})}.
\]
Now we can pull back $\wh\gamma_{\dt,\eps}^\pm$ to the Hadjidemetriou's rotating coordinates $(r,R,q_2,p_2)$ undoing changes $\Phi_{pot}$, $\Phi_{res}^{2}$ and $\Phi_{res}^{1}$:
\[
  \left\{
   \begin{aligned}
      r^\pm(t)&=1+\sqrt{\frac{\dt}{\eps}}\wh r^\pm(t),& \\
   R^\pm(t)&=\sqrt{\frac{\dt}{\eps}}\wh R^\pm(t),&\\
   q_2^\pm(t)&=\frac1{\eps^2}\Big(\rho^\pm(t)\cos\phi^\pm(t),\rho^\pm(t)\sin\phi^\pm(t)\Big),& \\
   p_2^\pm(t)&=\dt\eps\Big(\Upsilon^\pm(t)\cos\phi^\pm(t)-\frac{G^\pm(t)}{\rho^\pm(t)}\sin\phi^\pm(t), \Upsilon^\pm(t)\sin\phi^\pm(t)+\frac{G^\pm(t)}{\rho^\pm(t)}\cos\phi^\pm(t)\Big).& 
   \end{aligned}
   \right.
\]
{From} now on, we take $0<\eps\leq\eps_1$ and a stronger condition in $\dt$: $0<\dt<\dt_1\eps^{15}\ll\dt_1\eps^7$. 
As for the other two coordinates $(\theta,\Om)$, we already know that $\Om=\alpha=\mu(1-\mu)$ 
is fixed. 
Using the espression of $H_{rot}^\dt$ in (\ref{polar sys-4bp}), we obtain
 \be\label{theta-eq}
\dot\theta(t)=\frac{\partial H_{rot}^\dt}{\partial \Om}\Bigg|^{\Om=\alpha}_{\wh\gamma_{\dt,\eps}^\pm}=\frac{1-\dfrac{\dt G^\pm(t)}{\alpha\eps}}{\big(1+\sqrt{\dfrac{\dt}{\eps}}\wh r^\pm(t)\big)^2}=1-2\sqrt{\frac{\dt}{\eps}}\wh r^\pm(t)+\cO(\frac{\dt}{\eps}), 
\ee
of which the orbits 
\[
\gamma_{\dt}^\pm(t)=(r^\pm(t),\theta^\pm(t),R^\pm(t),\alpha,q_2^\pm(t),p_2^\pm(t)),\quad\forall t\in\R
\]
are still periodic, as long as  
\be\label{ration-cond}
\mathfrak R^\pm(\dt,\eps):=\frac{2\pi}{\int_0^{\wh T_{\dt,\eps}^\pm}\frac{\partial H_{rot}^\dt}{\partial \Om}\Big|^{\Om=\alpha}_{\wh\gamma_{\dt,\eps}} dt}\in\Q.
\ee
Observe that, by (\ref{theta-eq}), we can estimate
\[
\mathfrak R^\pm(\dt,\eps)=\frac{2\pi}{\wh T_{\dt,\eps}^\pm+\cO(\sqrt{\dfrac{{\dt}}{\eps}})}.
\]
By using the formula of $\wh T_{\dt,\eps}^\pm$ in (\ref{peri-res-orbit}), we can choose $0<\eps_0<\eps_1/2$ suitably small, such that for any $\eps\in[\eps_0/2,\eps_0]$ and  $0\leq\dt\leq \dt_1\eps^{15}$, previous ratio can be estimated by
\[
\mathfrak R^\pm(\dt,\eps)=1\mp\eps^3+\cO(\eps^7)+\cO(\sqrt{\dfrac{\dt}{\eps}}).
\]
So we can fix a rational number 
\be\label{restr-number}
\underbrace{\frac{q-1}{q}\in[1-\frac{3}{4}\eps_0^3,1-\frac{3}{8}\eps_0^3]}_{\text{for }\mathfrak R^+},\quad\quad\Big(\text{resp.}\ \underbrace{\frac{q+1}{q}\in[1+\frac{3}{8}\eps_0^3,1+\frac{3}{4}\eps_0^3]}_{\text{for }\mathfrak R^-}\Big)
\ee
such that for any $0\leq\dt\leq\dt_0$ with $\dt_0:=\dt_1\big(\frac{\eps_0}{2}\big)^{15}$, we can always find a $\eps_{\dt}^{\pm}\in[9\eps_0/16, 15\eps_0/16]$, such that 

%
%
\[
\mathfrak R^+(\dt,\eps_{\dt}^+)=\frac{q-1}q \hspace{50pt} (\text{ resp. } \mathfrak R^-(\dt,\eps_{\dt}^-)=\frac{q+1}q).
\]
Based on this choice, the orbit $\gamma_{\dt}^\pm$ must be periodic, with the period 
\be \label{eq:period}
T_{\dt}^\pm=(q\mp 1)\wh T_{\dt,\eps}^\pm&=&\frac{2\pi(q\mp 1)}{1\mp\eps_{\dt}^{\pm3}+\cO(\eps_{\dt}^{\pm7})+\cO(\sqrt{\dt/{\eps_{\dt}^\pm}})}\nonumber\\
&=&2\pi q\big(1+\cO(\sqrt{\frac{\dt}{\eps_{\dt}^\pm}})\big).
\ee
Recall that $\dt_0=\dt_1(\frac{\eps_0}{2})^{15}$ and due to (\ref{restr-number}),  for all $0\leq\dt\leq\dt_0$ we have
\[
\frac{4\pi}{3}\eps_0^{-3}\leq T_{\dt}^\pm\leq\frac{16\pi}{3}(2+\eps_0^{-3})
\]
as long as $0<\eps_0\ll1$.
As for (\ref{rot-conti-est}), it's an instant deduction due to estimates \eqref{est-peri-res} obtained in  Proposition \ref{prop-conti-peri}.
\end{proof}
\vspace{10pt}



\subsection{Continuation method from RPC3BP to 3BP}\label{s2s2}

This section is devoted to proof the existence of periodic orbits of the Hamiltonian system $
H_{res}^{\dt, \eps}$ given in \eqref{eq:ham-res-wh}
in the domain $\cD_{res}$ defined in \eqref{eq:Dres}.

\begin{prop}\label{prop-conti-peri}
There exist $\eps_1=\eps_1(\cD_{res},\mu)\ll1$ and $\dt_1=\dt_1(\cD_{res},\mu)$, such that for any $0<\eps\leq\eps_1$ and $0<\dt\leq\dt_1\eps^{7}$, we can find two periodic orbits 
$\wh \gamma_\dt^\pm$ 
of system 
$
H_{res}^{\dt, \eps}$ in (\ref{eq:ham-res-wh}) 
\[
 \wh\gamma_{\dt, \eps}^\pm(t):=\Big\{\Big(\wh r^\pm(t),\wh R^\pm(t),\rho^\pm(t),\Upsilon^\pm(t),\phi^\pm(t),G^\pm(t)\Big)\Big\}\subset \cD_{res},
\]
where $\cD_{res}$ is defined in \eqref{eq:Dres}, 
which satisfy
\[
\wh\gamma_{\dt,\eps}^\pm(t+\wh T_{\dt,\eps})= \wh\gamma_{\dt,\eps}^\pm(t),\quad\forall \ t\in\R
\] 
with the period
 \be\label{peri-res-orbit}
 \wh T_{\dt,\eps}^\pm=\frac{2\pi}{1\mp\eps^3+\cO(\eps^7)+\cO(\sqrt{\dt/\eps})}.
 \ee
Moreover, there exist constants $M_1, M_3 > 1$ depending on $\eps_1,\dt_1$, such that
\be\label{est-peri-res}
 \left\{
   \begin{aligned}
      \|\wh r^\pm(t)\|&\leq M_3\frac{\sqrt\dt}{\eps^{7/2}},&\\
   \|\wh R^\pm(t)\|&\leq M_3\frac{\sqrt\dt}{\eps^{7/2}},&\\
   \|\rho^\pm(t)-1\|&\leq M_1(\mu\eps^4+\frac{\sqrt\dt}{\eps^{7/2}}),&\\
   \|\Upsilon^\pm(t)\|&\leq M_1(\mu\eps^4+\frac{\sqrt\dt}{\eps^{7/2}}),& \\
   \|\dot\phi^\pm(t)-(1\mp\eps^3)\|&\leq 2M_1(\mu\eps^7+\frac{\sqrt\dt}{\eps^{1/2}}),& \\
   \|G^\pm(t)\mp1\|&\leq M_1(\mu\eps^4+\frac{\sqrt\dt}{\eps^{7/2}}).&
   \end{aligned}
   \right.
\ee
\end{prop}

\begin{proof}
First observe (see Appendix \ref{A}) that, as the domain $\cD_{res}$ (see \eqref{eq:Dres}) is a compact set, 
 there exists a constant $M=M(\dt_1,\cD_{res})>0$ such that 
\begin{equation}\label{eq:boundDeltaHres}
\big\|\Dt  H_{res}^{\dt,\eps}(\wh r,\rho,\phi,\Upsilon,G)\big\|
=\sup_{\cD_{res}} |\Dt  H_{res}^{\dt,\eps} (\wh r,\rho,\phi,\Upsilon,G)| \leq M\sqrt{\frac{\dt}{\eps}}.
\end{equation}
Therefore  by removing 
$\cO(\frac{\mu\eps^7}{\rho^3})+\Dt H_{res}^{\dt,\eps}(\wh r,\rho,\phi,\Upsilon,G)$, from the Hamiltonian $H_{res}^{\dt,\eps}$ in \eqref{eq:ham-res-wh}
we get a decoupled truncated system 
\be
\ol H_{res}^0(\wh r,\wh R,\rho,\Upsilon,\phi,G)=-G+\eps^3\Big[\frac12(\Upsilon^2+\frac{G^2}{\rho^2})-\frac1\rho\Big]+\frac12\Big[\frac{\wh R^2}{\alpha}+\alpha\wh r^2\Big],
\ee
of which two periodic solutions can be found:
\be\label{unper-peri-orbit}
\ol\gamma_0^\pm:=\{\wh r=0,\ \wh R=0,\ \rho=1,\ \Upsilon=0,\ \phi\in\T,\  G=\pm 1\}.
\ee
The period of $\ol\gamma_0^\pm$ is $\frac{2\pi}{1\mp\eps^3}$. 
Notice that $\ol\gamma_0^\pm$ lie on the energy level $\{\ol H_{res}^0=E^\pm\}$ respectively,  with  $E^\pm:=\mp1-\frac{\eps^3}2$.
 Notice that the energy level can be expressed as a graph
\be
G&=&\ol G^\pm(\rho,\Upsilon,\wh r,\wh R,E^\pm;\eps)\nonumber\\
&=&-E^\pm+\frac12\Big[\frac{\wh R^2}{\alpha}+\alpha\wh r^2\Big]+\eps^3 \big[\frac{\Upsilon^2}{2}+\frac{(\frac{\wh R^2}{\alpha}+\alpha\wh r^2-2E^\pm)^2}{8\rho^2}-\frac{1}{\rho}\big]+\cO(\eps^6).\nonumber
\ee
If we further restrict the energy level to the section $\Sigma_0:=\{\phi=0\}$ and consider the 
Poincar\'{e} maps $\ol {\cP}_0^{\pm}:\Sigma_0\rightarrow\Sigma_0$, we can see that it equals just the time-$2\pi$ map of the following ODE (rectified flow):
\be\label{section-flow}
\quad\quad\begin{pmatrix}
  \wh r'\\
   \wh R'\\
   \rho'       \\
   \Upsilon'   
\end{pmatrix}=
\begin{pmatrix}
 \dfrac{\partial \wh r}{\partial\phi}  \vspace{5pt} \\
        \dfrac{\partial \wh R}{\partial\phi} \vspace{5pt} \\
      \dfrac{\partial \rho}{\partial\phi}  \vspace{5pt}  \\
     \dfrac{\partial \Upsilon}{\partial\phi}  
\end{pmatrix}=
\underbrace{
\begin{pmatrix}
\dfrac{\frac{\partial \ol H_{res}^0}{\partial \wh R}}{\frac{\partial\ol H_{res}^0}{\partial G}}  \vspace{5pt} \\
        \dfrac{-\frac{\partial \ol H_{res}^0}{\partial \wh r}}{\frac{\partial\ol H_{res}^0}{\partial G}} 
\vspace{5pt}\\
      \dfrac{\frac{\partial \ol H_{res}^0}{\partial \Upsilon}}{\frac{\partial\ol H_{res}^0}{\partial G}}  \vspace{5pt}  \\
     \dfrac{-\frac{\partial \ol H_{res}^0}{\partial \rho}}{\frac{\partial\ol H_{res}^0}{\partial G}}  
       \end{pmatrix}
=
\begin{pmatrix}
  \dfrac{\wh R}{\alpha(-1+\eps^3\dfrac{\ol G^\pm}{\rho^2})}\\
       \dfrac{-\alpha \wh r}{(-1+\eps^3\dfrac{\ol G^\pm}{\rho^2})}\\
      \dfrac{\eps^3\Upsilon}{-1+\eps^3\dfrac{\ol G^\pm}{\rho^2}}    \\
      \dfrac{\eps^3(\dfrac{\ol G^{\pm2}}{\rho^3}-\dfrac{1}{\rho^2})}{-1+\eps^3\dfrac{\ol G^\pm}{\rho^2}} 
\end{pmatrix}
}_{
:=\mathscr V^\pm(\wh r, \wh R,\rho,\Upsilon,\phi)}=\begin{pmatrix}
       -\frac{\wh R}{\alpha} +\cO(\eps^3)\\
       \alpha\wh r+\cO(\eps^3)\\
      \cO(\eps^3)    \\
      \cO(\eps^3) 
\end{pmatrix}.
   \ee
Therefore, the periodic orbits $\ol\gamma_0^\pm$ correspond to  fixed points of $\ol\cP_0^\pm$, i.e. $Z^*:=(\wh r,\wh R,\rho,\Upsilon)=(0,0,1,0)$. 
Linearizing $\ol\cP_0 ^\pm$ around the fixed point, we know
\be
D\ol\cP_0^\pm(Z^*)=e^{2\pi D\mathscr V^\pm (\ol\gamma_0^\pm)}&=&\exp\Bigg\{2\pi\begin{pmatrix}
 0&\frac{1}{\alpha(-1\pm\eps^3)}  &0&0  \\
 \frac{-\alpha}{-1\pm\eps^3}&0  &0&0\\
   0&0& 0   & \frac{\eps^3}{-1\pm\eps^3}\\
   0&0& \frac{-\eps^3}{-1\pm\eps^3}   &  0
\end{pmatrix}\Bigg\}\nonumber\\
&=&\begin{pmatrix}
 \cos\frac{2\pi}{-1\pm\eps^3}&\frac1\alpha\sin\frac{2\pi}{-1\pm\eps^3}&0&0 \\
    -\alpha\sin\frac{2\pi}{-1\pm\eps^3}& \cos\frac{2\pi}{-1\pm\eps^3}  &0&0\\
     0&0&\cos\frac{2\pi\eps^3}{-1\pm\eps^3}     &   \sin\frac{2\pi\eps^3}{-1\pm\eps^3}  \\
     0&0&-\sin\frac{2\pi\eps^3}{-1\pm\eps^3}   &   \cos\frac{2\pi\eps^3}{-1\pm\eps^3}
\end{pmatrix},
\ee
%
%
of which we can solve the multipliers by 
\[
e^{i\frac{2\pi\eps^3}{-1\pm\eps^3}},\ 
e^{-i\frac{2\pi\eps^3}{-1\pm\eps^3}},\ 
e^{i\frac{2\pi}{-1\pm\eps^3}}=e^{\pm i\frac{2\pi \eps ^3}{-1\pm\eps^3}}, \ 
e^{-i\frac{2\pi}{-1\pm\eps^3}}=e^{\mp i\frac{2\pi \eps ^3}{-1\pm\eps^3}}.
\]
Estimating previous multipliers by the Taylor expansion, all of them can be estimated by $1\pm2\pi\eps^3i+\cO(\eps^6)$. 
That implies $I-D\ol\cP_0^\pm (Z^*)$ is invertible and 
\be\label{matrix-close to id}
\|(I-D\ol\cP_0 ^\pm (Z^*))^{-1}\|\leq \frac{2}{\eps^3}.
\ee
Another fact due to (\ref{section-flow}) is that 
\be\label{zero-poin-est}
\|\ol\cP_0(Z)-Z\|_{C^2}\leq M_0\eps^3,\quad\forall Z\in B(Z^*,C_{Z^*})\cap\Sigma_0,
\ee
where $B(Z^*,C_{Z^*})\subset\cD_{res}$ i a ball centered at $Z^*$ of radious $C_{Z^*}$, and $M_0$ is a constant depending on $C_{Z^*}$.
 We will use these conditions in the following computation.

As $
H_{res}^{\dt,\eps} (\wh r,\wh R,\rho,\Upsilon,\phi,G)=\ol H_{res}^0+\cO(\mu\eps^7)+\cO(\sqrt{\dfrac{\dt}{\eps}})$ in the domain $\cD_{res}$ (see \eqref{eq:ham-res-wh} and \eqref{eq:boundDeltaHres}),  restricted to certain domain $B(Z^*,C_{Z^*})\cap\Sigma_0$, the associated Poincar\'e map $\wh\cP_\dt:\Sigma_0\rightarrow\Sigma_0$
should satisfy
\[
\wh\cP_\dt ^\pm :=\ol\cP_0^\pm +\cP_1^\pm 
\]
with 
\be\label{reminder-poincare}
\|\cP_1^\pm \|_{C^2}\leq M_1(\mu\eps^7+\sqrt{\frac{\dt}{\eps}})
\ee
for some constant $M_1=M_1(C_{Z^*},\mu)$. 
Let $\sigma=4M_1(\mu\eps^4+\sqrt\dt\eps^{-7/2})$ and we try to find a fixed point of 
$\wh\cP^\pm _\dt$ in $B(Z^*,\sigma)\cap\Sigma_0$, which is equivalent to find a point $\cZ\in B(0,\sigma)\cap\Sigma_0$, such that 
\be
\cZ=\mathscr F^\pm (\cZ):=[I-D\ol\cP_0^\pm (Z^*)]^{-1}\cdot\big[\cQ^\pm (Z^*,\cZ)+\cP_1^\pm (Z^*+\cZ)\big]
\ee 
with
\[
\cQ^\pm (Z^*,\cZ):=\ol\cP_0^\pm (Z^*+\cZ)-Z^*-D\ol\cP_0^\pm (Z^*)\cZ.
\]
Notice that by (\ref{matrix-close to id}) and (\ref{reminder-poincare}) we have
\[
|\mathscr F ^\pm (0)|\leq 2M_1 (\mu\eps^4+\sqrt{\frac{\dt}{{\eps^7}}}),
\]
and $\sigma=2|\mathscr F^\pm (0)|$. 
Due to (\ref{zero-poin-est}), there exists a constant $M_2=M_2(C_{Z^*})$ such that 
\[
\|\cQ^\pm  (Z^*,\cZ)\|_{C^2}\leq M_2\eps^3,\quad\forall \cZ\in B(0,\sigma).
\]
Accordingly, we have 
\be
Lip(\mathscr F^\pm )|_{B(0,\sigma)}&\leq& [I-D\ol\cP_0^\pm (Z^*)]^{-1}   \cdot   
\big[ \|\cQ^\pm  \|_{C^2}\cdot \|\cZ\|
+\|\cP_1^\pm (Z^*+\cZ)\|_{C^1}\big]\nonumber\\
&\leq&2M_2\sigma+2M_1(\mu\eps^4+\sqrt{\frac{\dt}{{\eps^7}}})\nonumber\\
&\leq& 2M_1(4M_2+1)\cdot(\eps^4+\sqrt{\frac{\dt}{{\eps^7}}}).
\ee
The Brouwer Fixed Point Theorem implies that once
\[
Lip(\mathscr F^\pm )|_{B(0,\sigma)}=2M_1(4M_2+1)\cdot(\eps^4+\sqrt{\frac{\dt}{{\eps^7}}})\leq1/2, 
\]
there must be a fixed point of $\mathscr F^\pm $ in $B(0,\sigma)$. 
So we can take $0<\eps\leq\eps_1$ and $0<\dt\leq\dt_1\eps^7$ with
\be
\dt_1=\frac{1}{64M_1^2(4M_2+1)^2},\quad\quad\eps_1=\frac{1}{2\sqrt[4]{M_1(4M_2+1)}} 
\ee
to ensure the Brouwer Fixed Point Theorem work. 
Accordingly, there exists $Z\in B(Z^*,\sigma)\cap\Sigma_0$, such that $\wh\cP_\dt^\pm ( Z)= Z$. 
The existence of the fixed point $ Z$ indicates the existence of a periodic orbit $\wh\gamma_{\dt,\eps}^\pm$ satisfying
\[
 \wh\gamma_{\dt,\eps}^\pm(t):=\Big\{\Big(\wh r^\pm(t),\wh R^\pm(t),\rho^\pm(t),\Upsilon^\pm(t),\phi^\pm(t),G^\pm(t)\Big)\Big\}
\]
(for different $\ol G^\pm$), of which the $3^{rd}, 4^{th}$ and $6^{th}$ inequality (\ref{est-peri-res}) holds. 
Recall that 
\[
\dot\phi\big|_{\wh\gamma_{\dt,\eps}^\pm}=-\frac{\partial
H_{res}^{\dt,\eps}}{\partial G}\Bigg|_{\wh\gamma_{\dt,\eps}^\pm}=1\mp\eps^3+\cO(\mu\eps^7)+\cO(\sqrt{\frac{\dt}{\eps}}),
\]
that implies the $5^{rd}$ inequality of (\ref{est-peri-res}) and the period of $\wh\gamma_{\dt,\eps}^\pm$ satisfies
\[
\wh T_{\dt,\eps}^\pm=\frac{2\pi}{1\mp\eps^3+\cO(\mu\eps^7)+\cO(\sqrt{\dfrac{\dt}{\eps}})}.
\]
Notice that for the $(\wh r,\wh R)-$component, using that the term $\mathcal \cO(\frac{\mu \eps ^7}{\rho^3})$ is independent of $\wh r$ and $\wh R$, we have
\[
\left\{
\begin{aligned}
\dot{\wh r}\big|_{\cD_{res}}&=\frac{\partial 
H_{res}^{\dt,\eps}}{\partial\wh R}=
\frac{\partial \ol H_{res}^0}{\partial\wh R}+\frac{\partial \Dt H_{res}^{\dt,\eps}}{\partial\wh R}
=\frac{\partial \ol H_{res}^0}{\partial\wh R}+\cO(\sqrt{\dfrac{\dt}{\eps}}),& \\
\dot{\wh R}\big|_{\cD_{res}}&=-\frac{\partial 
H_{res}^{\dt,\eps}}{\partial\wh r}
=-\frac{\partial \ol H_{res}^0}{\partial\wh r}-\frac{\partial \Dt  H_{res}^{\dt,\eps}}{\partial\wh r}
=-\frac{\partial \ol H_{res}^0}{\partial\wh r}+\cO(\sqrt{\dfrac{\dt}{\eps}}).& 
\end{aligned}
\right.
\]
That implies 
\[
\|\pi_{(\wh r,\wh R)}\cP_1^\pm\|_{C^2}\leq M_1\sqrt{\dfrac{\dt}{\eps}}
\] 
in the domain $B(Z^*,C_{Z^*})\cap\Sigma_0$. Moreover, we have 
\[
\big|\pi_{(\wh r,\wh R)}\mathscr F^\pm(0)\big|\leq 2M_1 \sqrt{\frac{\dt}{{\eps^7}}}.
\]
So for $\sigma':=M_3\sqrt{\dfrac{\dt}{{\eps^7}}}$ with $M_3$ a constant depending on $M_1,M_2$ (decided later), we get 
\[
Lip\big(\pi_{(\wh r,\wh R)}\mathscr F^\pm\big)|_{B(0,\sigma')} =2M_2\sigma'+2M_1\sqrt{\frac{\dt}{{\eps^7}}}.
\] 
For $0< M_3<\dfrac{16M_1^2(4M_2+1)^2-M_1}{M_2}$, the Brouwer Fixed Point Theorem implies that 
the $(\wh r,\wh R)-$ component of the fixed point $\wt Z$ lies in $B_{\R^2}(0,\sigma')$.  
This leads to the first two inequalities of (\ref{est-peri-res}) then we finally proved this Proposition.
\end{proof}

\begin{rmk}

For system $\ol H_{res}^0$, we have the freedom to choose different $\ol\gamma_0^\pm$, by taking different $G-$value, see (\ref{unper-peri-orbit}). These orbits can all be continued to $\wh\gamma_{\dt,\eps}^\pm$ for system $\wh H_{res}^\dt$. However, we should alwayd keep the $G-$value independent of $\eps$ to avoid the blowup of the period $\wh T_{\dt,\eps}$ as $\eps\rightarrow 0$. Let's point out that in \cite{X},the author uses  a periodic orbit with the period of $\cO(\eps^{-1})$, which is quite different from the mechanism we assumed.
\end{rmk}

  \vspace{20pt}


\section{Oscillatory otbits in the RP3BP}\label{s3}
\vspace{20pt}


This section is devoted to prove the second item of Theorem \ref{main-thm}. 
In last section we proved the existence of the comet type periodic orbits $\gamma^\pm_\dt$  for the 3BP, and claimed associated estimates on them (see Theorem \ref{rot-continue}).  
Now we add a massless Asteroid to the  previous 3BP system,  assume that the primaries move in one of these comet type periodic orbits and prove that the asteroid can have oscillatory motions.

\subsection{Setup of the RP4BP and the invariant manifold of infinity}.  \label{setup}

As one can choose any of the two  periodic orbits for the rest of the work, from now on we can just pick $\gamma_\dt^+$ which has the associated period $T_\dt^+$, and for brevity we remove the `$+$' (also for $\eps_\dt^+$). 
Recall that $\mu$ can have any  fixed value. 
Besides, due to Remark \ref{rmk:uni}, $\eps_\dt^+$ is uniformly bounded w.r.t. $\dt$, i.e. $\eps_\dt^+\sim O(1)$ as $\dt\rightarrow 0^+$.\\

Now we derive the RP4BP Hamiltonian as (\ref{ham-4bp}):
\be\label{eq:ham1}
 H_A^{\dt,\mu}(x_A,y_A,t)
&=&\frac12|y_A|^2-V_A^{\dt,\mu}(x_A,t), \\
& &\quad (x_A,y_A,t)\in\R^4\times\T, \quad 0<\mu \le \frac12,\ 0\le \dt \le \dt_0\nonumber
\ee
where $\dt_0$ is the value given in Theorem \ref{rot-continue}. 
We add the superscript `$\dt,\mu$' to indicate the dependence of $H_A^{\dt,\mu}$ about these parameters.\\

Due to (\ref{exp-comb-sym-tran}) and \eqref{rot-conti-est}, the potential function of (\ref{eq:ham1}) has an explicit expression:
\be\label{potent-4bp}
V_A^{\dt,\mu}(x_A,t)&=&\frac{1-\mu}{|x_A+\mu r(t)(\cos\theta(t),\sin\theta(t))+\frac{\dt}{1+\dt}e^{i\theta(t)}q_2(t)|}\nonumber\\
& &+\frac{\mu}{|x_A-(1-\mu) r(t)(\cos\theta(t),\sin\theta(t))+\frac{\dt}{1+\dt}e^{i\theta(t)}q_2(t)|}\nonumber\\
& &+\frac{\dt}{|x_A-\frac{1}{1+\dt}e^{i\theta(t)}q_2(t)|},
\ee
Observe that $V_A^{\dt,\mu}(x_A,t)$ is $T_{\dt}-$periodic in $t$. 
%
%

As system (\ref{eq:ham1}) is non-autonomous, we can consider the augmented autonomous Hamiltonian of three degrees os freedom
namely, we have 
\be\label{ext-ham}
\wt H_A^{\dt,\mu}(x_A,y_A,s,I)&:=& H_A^{\dt,\mu}(x_A,y_A,\frac{s}{\nu_\dt})+\nu_\dt I,\nonumber\\
&=&\frac12|y_A|^2-V_A^{\dt,\mu}(x_A,\frac s{\nu_\dt})+\nu_{\dt} I, \quad (x_A,y_A,s,I)\in\R^4\times\T\times\R
\ee
with $\nu_\dt=2\pi/T_{\dt}$, $s=\nu_\dt t$ and $I$ being the conjugated variable to $s$.
The benefit of doing this is that $\wt H_A^{\dt,\mu}$ becomes autonomous and $2\pi-$periodic of $s\in\T=\R/[0,2\pi]$.
In fact, as the added action variable $I$ does not play any role in the dynamics, we will always work in the energy level $\wt H_A^{\dt,\mu}=0$ and then ``ignore'' this variable. This reduction gives the so-called $5-$dimensional extended phase space and is equivalent to  just adding the equation $\dot s=\nu_\dt$ to the Hamiltonian equations of $H_A^{\dt,\mu}$.\medskip

Writing $\wt H_A^{\dt,\mu}$ in  polar coordinates:
%

\[
 \left\{
\begin{split}
x_A&=(\xi\cos\psi, \xi\sin\psi),\\
y_A&=(\Xi\cos\psi-\frac{\Psi}{\xi}\sin\psi, \Xi\sin\psi+\frac{\Psi}{\xi}\cos\psi)
\end{split}
\right.
\]
we obtain
\be
\wh {H}_{A}^{\dt,\mu}(\xi,\psi,\Xi,\Psi, s,I)&=&\nu_\dt I+\frac12(\Xi^2+\frac{\Psi^2}{\xi^2})-\frac{1+\dt}{\xi}-\wh{\Dt V}_A^{\dt,\mu}(\xi,\psi, \frac s{\nu_\dt}),
\ee
where, by \eqref{rot-conti-est},
\be\label{incre-pote}
\wh{\Dt V}_A^{\dt,\mu}(\xi,\psi, \frac s{\nu_\dt}):=V_A^{\dt,\mu}(\xi e^{i\psi},\frac s{\nu_\dt})-\frac{1+\dt}{\xi}\sim \cO(\frac {\mu+\dt}{\xi^3}).
\ee



 For $\xi\gg1$, we consider the McGehee transformation by setting $\xi=2/x^2$ with $x>0$, then 
\[
d\xi\wedge d\Xi+d\psi\wedge d\Psi=-4x^{-3}dx\wedge d\Xi+d\psi\wedge d\Psi.
\]
That means the Hamiltonian in the new variables is 
\be\label{ham-4bp-polar}
\wt \cH_{A}^{\dt,\mu}(x,\psi,\Xi,\Psi,s,I)=\nu_\dt I+\frac12(\Xi^2+\frac{\Psi^2x^4}4)-\frac{(1+\dt)x^2}2-\wt {\Dt V_A}^{\dt,\mu}
(x,\psi,s)
\ee
with 
\be\label{distance-potent}
\wt {\Dt V_A}^{\dt,\mu}(x,\psi,s):=\wh{\Dt V_A}^{\dt,\mu}(2/x^2,\psi, \frac s{\nu_\dt}) \sim \cO((\mu+\dt) x^6)
\ee
and the associated ODE is 
\[
\dot x=\frac{-x^3}{4}\partial_\Xi\wt \cH_A^{\dt,\mu},
\quad
\dot \Xi=\frac{x^3}{4}\partial_x\wt \cH_A^{\dt,\mu},\quad
 \dot \psi=\partial_\Psi\wt \cH_A^{\dt,\mu},
\quad 
\dot \Psi=-\partial_\psi\wt \cH_A^{\dt,\mu},\quad\dot s=\nu_\dt.
\]
Using the form of the potential in \eqref{distance-potent}, we can estimate previous ODE by
\be\label{eq-infity}
\left\{
\begin{aligned}
\dot x&=-\frac14 x^3\Xi,& \\
\dot\Xi&=-\frac{1+\dt}{4}x^4+\cO((\mu+\dt) x^6),& \\
\dot\psi&=\frac{1}{4}\Psi x^4,& \\
\dot\Psi&= x^6\beta^{\dt,\mu}(\psi,s)+\cO((\mu+\dt)x^8 ),& \\
\dot s&=\nu_\dt.&
\end{aligned}
\right.
\ee
where $\beta^{\dt,\mu}(\psi,s)$ is a periodic function defined on 
$\T\times\T$ and $\beta^{\dt,\mu}\sim \cO(\mu+\dt)$.


In view of this, the ``parabolic infinity'' $x=\Xi=0$  is foliated by the parabolic periodic orbits
\[
\wt\Lb_{\psi_0,\Psi_0}^{\dt,\mu}:=\Big\{(x,\psi,\Xi,\Psi,s)\in\R\times\T\times\R\times\R\times\T \,  \Big| \ x=\Xi=0,\ \psi=\psi_0,\ \Psi=\Psi_0\Big\}.
\]
Besides, as $\dot s= \nu_\dt=\frac{2\pi}{T_{\dt}}$,
for any fixed $\psi_0$ and $\Psi_0$, if we denote by
$\wt\phi_t^{\dt,\mu}$  the flow of the equation \eqref{eq-infity}, we have that:
\[
\wt\phi_t^{\dt,\mu}(0,\psi_0,0,\Psi_0,s)=(0,\psi_0,0,\Psi_0,s+\nu_\dt t)
\]
therefore,  $\wt\Lb_{\psi_0,\Psi_0}^{\dt,\mu}$ is a $T_{\dt}^+$-periodic orbit.
\\

Even if these periodic orbits are parabolic, next theorem gives that they have  stable (resp. unstable) 2-dimensional
invariant manifolds $W^s(\wt\Lb_{\psi_0,\Psi_0}^{\dt,\mu})$ (resp. $W^u(\wt\Lb_{\psi_0,\Psi_0}^{\dt,\mu})$).


\begin{thm}\cite{GMSS}\label{stable-mani-thm}
Fix any 
$(\psi_0,\Psi_0)\in\T\times\R$. 
There exists $\rho_0\in\R_+$ and $0< \dt_0^*\leq\dt_0$, where $\dt_0$ is given in Theorem \ref{rot-continue}, such that for any $0<\rho<\rho_0$, $0\leq\dt<\dt_0^*$, the local stable set 
\be
W^s(\wt\Lb_{\psi_0,\Psi_0}^{\dt,\mu})&:=&
\Big\{(x,\psi,\Xi,\Psi,s)\in\R\times\T\times\R\times\R\times\T\Big|\ 
\pi_x\wt\phi_t^{\dt,\mu}(x,\psi,\Xi,\Psi, s)>0,\nonumber\\
& &\big|\pi_{(x,\Xi)}\wt\phi_t^{\dt,\mu}(x,\psi,\Xi,\Psi, s)\big|<\rho, 
\lim_{t\rightarrow+\infty}\rm{dist}(\wt\phi_t^{\dt,\mu}(x,\psi,\Xi,\Psi, s),\wt\Lb_{\psi_0,\Psi_0}^{\dt,\mu})=0\Big\}\nonumber
\ee 
is a 2-dimensional $C^\infty$   manifold. 
Moreover, it is analytic for $x>0$ and   depends analytically on $(\psi_0,\Psi_0,\dt)\in\T\times\R$.
The analogous result for the (local) unstable set $W^u(\wt\Lb_{\psi_0,\Psi_0}^{\dt,\mu})$ holds as well.
\end{thm}

The proof of this theorem is  analogous to the one of Theorem $2.1$ in \cite{GMSS}. 
Observe that if we make the following change of variables 
\begin{equation}\label{eq:changepq}
q=\frac12(\sqrt{1+\dt}x-\Xi),\quad p=\frac12(\sqrt{1+\dt} x+\Xi),
\quad \alpha=(1+\dt) \psi+ \Psi \Xi,
\end{equation}
system \eqref{eq-infity}  becomes
\begin{equation}\label{eq:systempq}
 \left\{
\begin{split}
\dot q&=\frac{(q+p)^3}{4(1+\dt)}\Big(q+(q+p)^3 \cO_0\Big),\\
\dot p&=-\frac{(q+p)^3}{4(1+\dt)}\Big(p+(q+p)^3 \cO_0\Big),\\
\dot\alpha &=(q+p)^6\cO_0,\\
\dot\Psi&=(q+p)^6\cO_0,\\
\dot s&=\nu_\dt.
\end{split}
\right.
\end{equation}
This system has the form of system (14) in \cite{GMSS}. 
Therefore, Proposition 3 in that paper can be applied  giving  the existence and regularity  of the stable and unstable manifolds of the sets 
$$
\{ q=p=0, \ \alpha=\alpha_0, \ \Psi=\Psi_0, \ s\in \T\}
$$
for any $\alpha_0$, $\Psi_0$. 
Going back to variables $(x, \Xi,\psi,\Psi)$ we obtain the stable and unstable manifolds of 
$\wt\Lb_{\psi_0,\Psi_0}^{\dt,\mu}$. 
%
%
\\

As is shown in Theorem \ref{stable-mani-thm}, the points which tend asymptotically in forward (resp.  backward) time
to the periodic orbit $\wt\Lb_{\psi_0,\Psi_0}^{\dt,\mu}$ form a $2-$dimensional manifold $W^s(\wt\Lb_{\psi_0,\Psi_0}^{\dt,\mu})$ (resp. $W^u(\wt\Lb_{\psi_0,\Psi_0}^{\dt,\mu})$). 
The fact that the periodic orbit $\wt\Lb_{\psi_0,\Psi_0}^{\dt,\mu}$ is not hyperbolic but parabolic, makes its invariant manifold $W^s(\wt\Lb_{\psi_0,\Psi_0}^{\dt,\mu})$ (resp. $W^u(\wt\Lb_{\psi_0,\Psi_0}^{\dt,\mu})$) to be only $C^\infty$ at $\wt\Lb_{\psi_0,\Psi_0}^{\dt,\mu}$, although analytic at any other point and also analytic respect to $\psi_0,\Psi_0,\dt$.\\

When $\dt=\mu=0$, system (\ref{ham-4bp-polar}) becomes the Kepler system which is totally integrable. 
Therefore, the associated invariant manifolds $W^{u,s}(\wt\Lb_{\psi_0,\Psi_0}^{0,0})$ coincide and form a two-parameter family of parabolas in the configuration space.
Indeed, for $\dt=\mu=0$, as by \eqref{eq:period} $\nu_\dt =1/q$  our extended Hamiltonian \eqref{ham-4bp-polar} is the following:
\be
\wt\cH_A^{0,0}(x,\psi,\Xi,\Psi,s,I)= I/q+\frac12(\Xi^2+\frac{\Psi^2x^4}4)-\frac{x^2}2, 
\ee
so $I$ is a first integral and $\dot s=\nu_{\dt=0}=1/q$ becomes a free equation independent of the motion for the rest variables. Therefore, we can still get formulas of the homoclinic manifolds as in \cite{GMS}, and we exhibit them here:
%
\begin{equation}\label{3bp-para-lem}
\left\{
\begin{aligned}
x_h(t,\Psi_0)&=\frac{2}{\Psi_0\sqrt{1+\tau^2}},& \\
\Xi_h(t,\Psi_0)&=\frac{2\tau}{\Psi_0(1+\tau^2)},& \\
\psi_h(t,\Psi_0)&=\psi_0+\alpha_h(t), \quad \alpha_h(t)=2\arctan\tau,& \\
\Psi_h(t,\Psi_0)&=\Psi_0,& \\
s_h(t,s_0)&=s_0+ t/q& \\
\end{aligned}
\right.
\end{equation}
where $s_0\in\T$ is a free parameter and $\tau$ is a parametrization of $t$ through
\[
t=\frac{\Psi_0^3}{2}\big(\tau+\frac{\tau^3}3\big).
\]

\subsection{The case $\dt=0$: The RPC3BP}\label{s3s1}
 Through this subsection we assume that $\dt=0$ but $\mu\neq0$. 
 Notice that the primaries S-J-P form a RPC3BP of which we can find a periodic orbit $\gamma_*^+$ with the period $T_*^+=2\pi q$ (see Remark \ref{rmk:uni}). 
Besides,  since P has no attraction to A, we have a new RPC3BP of the system S-J-A.
%
Notice that $\nu_{\dt=0}=1/q$, then due to (\ref{theta-eq}) we have that the extended Hamiltonian \eqref{ham-4bp-polar} becomes (see \eqref{distance-potent}, \eqref{incre-pote} \eqref{potent-4bp}):
\be\label{dt=0-ham}
\wt H_{A}^{0,\mu}(x,\psi,\Xi,\Psi,s,I)&=& 
\frac Iq +\frac12(\Xi^2+\frac{\Psi^2x^4}{4})-\frac{x^2}{2}
-\wt V_A^{0,\mu}(x,\psi,s)\nonumber\\
&=&
\frac Iq +\frac12(\Xi^2+\frac{\Psi^2x^4}{4})-\frac{x^2}{2}\nonumber\\
& &-\underbrace{\Big[\frac{1-\mu}{| \frac2{x^2}e^{i\psi}+\mu (\cos qs,\sin qs)|}+\frac{\mu}{|\frac{2}{x^2} e^{i\psi}+(\mu-1) (\cos qs,\sin qs)|}-\frac{x^2}{2}\Big]}_{\wt V_A(x,\psi,s,\dt=0)}\nonumber\\
&=& \frac Iq +\frac12(\Xi^2+\frac{\Psi^2x^4}{4})-\frac{1-\mu}{|\frac2{x^2} e^{i(\psi-qs)}+(\mu,0)|}-\frac{\mu}{|\frac2{x^2} e^{i(\psi-qs)}+(\mu-1,0)|}.
\ee
This is the Hamiltonian of the RPC3BP, and as it is well known,  $\wt V_A^{0,\mu}(x,\psi,s)$ is a function of $x$ and $\psi-qs$. This is reflected in the fact that the system has a first integral, 
\[
\frac12(\Xi^2+\frac{\Psi^2x^4}{4})-\Psi-\frac{1-\mu}{|\frac2{x^2} e^{i(\psi-qs)}+(\mu,0)|}-\frac{\mu}{|\frac2{x^2} e^{i(\psi-qs)}+(\mu-1,0)|},
\]
which is actually the Jacobi constant.
%
%

Now gathering all the periodic orbits $\wt\Lb_{\psi_0,\Psi_0}^{0,\mu}$ with $\Psi_0 $ greater than a given $\Psi_1>0$, we get an invariant set 
\[
\wt\Lb_{[\Psi_1,+\infty)}^{0,\mu}:=\bigcup_{\psi_0\in\T,\Psi_0\geq \Psi_1}\wt\Lb_{\psi_0,\Psi_0}^{0,\mu}
\]
which is a {\it normally parabolic} 3-dimensional invariant manifold. 
The associated 4-dimensional stable (resp. unstable) manifolds can be defined by 
\[
W^\varsigma(\wt\Lb_{[\Psi_1,+\infty)}^{0,\mu}):=\bigcup_{\psi_0\in\T,\Psi_0\geq \Psi_1}W^\varsigma(\wt\Lb_{\psi_0,\Psi_0}^{0,\mu}),\quad\varsigma=u,\ s.
\]
Theorem 2.2 of \cite{GMS} implies that, when $\dt=0$ and $\mu\in(0,1/2]$, there exists $\Psi^*\gg1$ such that for any $\Psi_1\geq \Psi^*$, the invariant manifolds $W^s(\wt\Lb_{[\Psi_1,+\infty)}^{0,\mu})$ and $W^u(\wt\Lb_{[\Psi_1,+\infty)}^{0,\mu})$ intersect transversally in the whole $5-$dimensional space along two different $3-$dimensional homoclinic manifolds $\tilde \Gamma^\pm_0$.

More concretely,  consider the {\it Poincar\'e function}
\be\label{poincare-fun}
L^{0}(\psi_0,\Psi_0, s_0,\sigma):=\int_{-\infty}^{+\infty}\wt V_A^{0,\mu}(x_h(\sigma+t,\Psi_0), \psi_0+\alpha_h (\sigma+t,\Psi_0), s_0+t/q)dt
\ee
where 
$(x_h,\alpha_h)$ are components of the parameterization of the unperturbed separatrix given in \eqref{3bp-para-lem}. 

Using that the potential $\wt V_A^{0,\mu}(x,\psi,s)$ is a function of $x$ and $\psi-qs$ and changing the variables to $r=\sigma+t$ in the integral, one easily obtains that the potential $L^0$ satisfies: 
\[
L^0(\psi_0,\Psi_0,s_0,\sigma)=L^0(\psi_0-qs_0,\Psi_0,0,\sigma)=L^0(\psi_0-qs_0+\sigma,\Psi_0,0,0).
\]
Besides, \cite{DKRS} and \cite{GMS} also show that the Fourier expansion of $L^0(\psi_0-qs_0+\sigma,\Psi_0,0,0)$ contains only cosines of $\psi_0-qs_0+\sigma$, so,  for any $(\psi_0,s_0)\in\T^2$ , we can easily solve two critical points $\sigma^*_\pm\in\T$ given by 
\[
\sigma^*_-=qs_0-\psi_0, \quad \sigma^*_+=\pi+qs_0-\psi_0.
\]
%
%

The results in  \cite{GMS}, give that, for any $ 0< \mu \le \frac12$, if 
$\Psi_1\ge \Psi^*$ big enough, associated to the  zeros $\sigma^*_\pm$, there exist two transversal  intersections
%
%
%
between $W^s(\wt\Lb_{[\Psi_1,+\infty)}^{0,\mu})$ and $W^u(\wt\Lb_{[\Psi_1,+\infty)}^{0,\mu})$ along two homoclinic manifolds 
 $\wt \Gamma_0^\pm$ 
\[
\wt \Gamma_0^\pm \subset W^s(\wt\Lb_{[\Psi_1,+\infty)}^{0,\mu})\pitchfork W^u(\wt\Lb_{[\Psi_1,+\infty)}^{0,\mu})
\]
In fact, one can easily see that $\wt \Gamma_0^\pm $ are submanifolds diffeomorphic to $\wt\Lb_{[\Psi_1,+\infty)}^{0,\mu}$, which also satisfy 
\[
T_zW^s(\wt\Lb_{[\Psi_1,+\infty)}^{0,\mu})\bigcap T_zW^u(\wt\Lb_{[\Psi_1,+\infty)}^{0,\mu})=T_z\Gamma_0^\pm,\quad\forall z  \in  \wt \Gamma_0^\pm.
\]
and therefore, following \cite{DLS2},  we call them homoclinic chanels.


Associated to each of these channels $\wt \Gamma_0^\pm $, we can define global scattering maps
\[
\wt\cS_0^\pm: \wt\Lb_{[\Psi_1,+\infty)}^{0,\mu}\rightarrow \wt\Lb_{[\Psi_1,+\infty)}^{0,\mu}
\]
which associate to any point $(\psi,\Psi,s) \in \wt\Lb_{[\Psi_1,+\infty)}^{0,\mu}$ the point 
$\wt\cS_0^\pm(\psi,\Psi,s) \in \wt\Lb_{[\Psi_1,+\infty)}^{0,\mu}$ if there is an heteroclinic connection between 
these two points
through $\wt \Gamma_0^\pm$.
Moreover, \cite{DLS2} provides  formulas for these maps:
\[
\wt\cS_0^\pm(\psi,\Psi,s)=(\psi+\frac{\partial}{\partial \Psi} \mathcal{S}^\pm (\psi,\Psi,s), \Psi-\frac{\partial}{\partial \psi}\mathcal{S}^\pm (\psi,\Psi,s),s)
\]
where the functions  $\mathcal {S}^\pm$ are given, in first order, by $\mathcal{L}^\pm (\psi,\Psi,s)=L^0(\psi,\Psi,s,\sigma^*_\pm)$.
Next proposition in \cite{GMSS}, whose proof is straightforward using the computations in 
\cite{DKRS}, gives an asymptotic formula for the scattering maps of the RPC3BP:

\begin{prop}\cite{GMSS}
Let $\Psi_1>\Psi^*$, the scattering maps $\wt\cS_0^\pm: \T\times[\Psi_1,+\infty)\times\T\rightarrow\T\times[\Psi_1,+\infty)\times\T$ are of the form
\[
\wt\cS_0^\pm(\psi,\Psi,s)=(\psi+f^\pm(\Psi),\Psi,s)
\]
where
\be\label{scat-fun}
f^\pm(\Psi)=-\mu(1-\mu)\frac{3\pi}{2\Psi^4}+\cO(\Psi^{-8}).
\ee
\end{prop}

Next step is to study the RP4BP as a $\dt$-perturbation of the RPC3BP. To this send, in order to reduce the dimension  of the system we will work with  the $4$-dimensional stroboscopic Poincar\'e map. 
We choose a section $\Sigma=\{s=s_0\}$ and consider
\be\label{poin-3bp}
\cP_0:\Sigma\rightarrow\Sigma\quad\text{via }\ (x,\Xi,\psi,\Psi,s_0)\rightarrow 
\wt\phi_{T_{*}^+=2\pi q}^{0,\mu}(x,\Xi,\psi,\Psi, s_0)
\ee
Then $\Lb_{\psi_0,\Psi_0}^{0,\mu}:=\wt\Lb_{\psi_0,\Psi_0}^{0,\mu}\bigcap\Sigma$ becomes a two parameter family of parabolic fixed points of $\cP_0$. 
Each fix point has  1-dimensional stable (resp. unstable) manifold
\[
W^\varsigma(\Lb_{\psi_0,\Psi_0}^{0,\mu}):=W^\varsigma(\wt\Lb_{\psi_0,\Psi_0}^{0,\mu})\cap\Sigma,\quad\varsigma=u,\ s.
\]
Analogously, 
$\Lb_{[\Psi_1,+\infty)}^{0,\mu}=\wt\Lb_{[\Psi_1,+\infty)}^{0,\mu}\cap\Sigma$ 
is the 2-dimensional normally parabolic invariant cylinder of infinity with 3-dimensional invariant stable (resp. unstable) manifolds
\[
W^\varsigma(\Lb_{[\Psi_1,+\infty)}^{0,\mu})=W^\varsigma(\wt\Lb_{[\Psi_1,+\infty)}^{0,\mu})\cap\Sigma,\quad\varsigma=u,\ s
\]
which intersect transversally along two 2-dimensional homoclinic channels
$\Gamma_0^\pm =\wt \Gamma_0^\pm \cap \Sigma$. 
The two scattering maps associated to these homoclinic channels are given by
\be\label{poin-scat-fun-3bp}
\cS_0^\pm(\psi,\Psi)=(\psi+f^\pm(\Psi),\Psi)
\ee where $f^\pm$ is the same function in (\ref{scat-fun}). 
Recall that $s$ is a free variable in the formula of $\wt\cS_0^\pm$, so the definition of $\cS_0^\pm$ is independent of the choice of the section $\Sigma$.\\

\subsection{Scattering map of the RP4BP}\label{s3s2}

In this section we study the  general RP4BP, that is,  system (\ref{ham-4bp})  $0<\mu\leq 1/2$ and  $0<\dt\le \dt_0$. 
As we established at the end of previous section, we will work with the  stroboscopic Poincar\'e map 
associated to $\Sigma=\{s=s_0\}$, i.e.
\be\label{poin-4bp}
\cP :\Sigma\rightarrow\Sigma \quad\text{via } \   (x,\Xi,\psi,\Psi,s_0)\rightarrow \wt\phi_{T_{\dt}}^{\dt,\mu} (x,\Xi,\psi,\Psi, s_0).
\ee
and our goal is to apply perturbative arguments of $\cP$ respect to $\cP_0$ in (\ref{poin-3bp}) 
to establish the transversal intersection between the stable and unstable manifolds of the ``parabolic infinity'' $\Lb_{[\Psi_1,+\infty)}^{0,\mu}$. In fact, for our purposes, it is enough to consider a compact part of it. This will make the arguments simpler. 
Precisely, let $\Psi^2_0>\Psi^1_0>\Psi_1\ge \Psi^*$ be fixed.
Then,  formally 
$\cP=\cP_0+\dt\cP_1+\cO(\dt^2)$ in the restricted compact region $\{\Psi\in[\Psi_0^1,\Psi_0^2]\}$. 
So for the corresponding $\Lb_{[\Psi_0^1,\Psi_0^2]}^{\dt,\mu}=\Lb_{[\Psi_1,+\infty)}^{\dt,\mu}\bigcap \{\Psi\in[\Psi_0^1,\Psi_0^2]\}$,  
the stable and unstable manifolds of $\Lb_{[\Psi_0^1,\Psi_0^2]}^{\dt,\mu}$ intersect transversally for 
$\dt=\dt (\Psi_0^1,\Psi_0^2)>0$ small enough. 
This implies that there are two global homoclinic channels $\Gamma^\pm_\dt$ diffeomorphic and 
$\cO(\dt)-$close to $\Gamma^\pm_0$.\\

These two channels define two scattering maps
\[
\cS^\pm:\Lb_{[\Psi_0^1,\Psi_0^2]}^{\dt,\mu}\rightarrow \Lb_{[\Psi^*,+\infty)}^{\dt,\mu}, 
\]
depending regularly on $\dt$, i.e.
\be\label{poin-scat-fun-4bp}
\cS^\pm=\cS_0^\pm+\dt\cS_1^\pm+\cO(\dt^2)
\ee
where $\cS_0^\pm$ are the scattering maps of the RPC3BP given by (\ref{poin-scat-fun-3bp}). As is shown in Proposition 4 of \cite{DKRS}, the maps $\cS^\pm$ are area preserving maps on the cylinder $\Lb_{[\Psi_0^1,\Psi_0^2]}^{\dt,\mu}$.\\

Our goal is now to obtain an infinite sequence of fixed points $p_i=\Lambda_{\psi_i,\Psi_i} \in \Lb_{[\Psi_0^1,\Psi_0^2]}^{\dt,\mu}$   through the Poincar\'{e} map
connected  through heteroclinic orbits. 
Of course the sequence can be constant, and in this case we would have a fixed point $p$ with  an homoclinic connection, or finite, and this would give us a set of fixed points $p_1,\dots p_k$ connected through heteroclinic connections between them. 
The main observation here, as was established in \cite{GMSS} is that a  point $p$  with  an homoclinic orbit would correspond to a fix point $p$ of one Scattering map $\cS^\pm (p)=p$, a finite  heteroclinic chain of points  $p_1\dots p_k$  would correspond to a periodic orbit of the scattering map:
$(\cS^{\pm})^k (p_i)=p_{i}$, and an infinite sequence $\{p_i\}$ can be obtained if we find invariant curves of the scattering map. 
So, the dynamical study of this map will give the needed transition chain.
\\

Notice that $\cS^\pm$ are twist maps for $\dt$ sufficiently small.
In fact, formulas  \eqref{poin-scat-fun-3bp} show that, for $\Psi  \in [\Psi_0^1,\Psi_0^2]$, they satisfy a  twist condition if $0\le \dt$ small enough:
\[
\frac{\partial \cS^\pm}{\partial \Psi} \ge \frac{C}{(\Psi_0^2)^5}
\]
Therefore, one can apply the classical Twist Theorem of Herman \cite{He}, to obtain that there exist 
%
KAM curves of $\cS^\pm$ inside $\Lb_{[\Psi_1,\Psi_2]}^{\dt,\mu}$ associated to some diophantine numbers. 
Clearly, any orbit of $\cS^\pm$ on the KAM curve would be bounded and gives us a infinite sequence 
of points in $\Lb_{[\Psi_1,\Psi_2]}^{\dt,\mu}$ with heteroclinic orbits between them as wanted.
In terms of  the Poincar\'e map $\cP$ in (\ref{poin-4bp}), we have obtained a sequence of fixed points $\{\Lb_{\psi_k,\Psi_k}^{\dt,\mu}\}_{k\in\mathbb N}\subset\Lb_{[\Psi_1,\Psi_2]}^{\dt,\mu}$ 
such that 
$W^u(\Lb_{\psi_k,\Psi_k}^{\dt,\mu})$ intersects $W^s(\Lb_{[\Psi_1,\Psi_2]}^{\dt,\mu})$ transversally at a point belonging to $W^s(\Lb_{\psi_{k+1},\Psi_{k+1}}^{\dt,\mu})$. \\

\begin{rmk}\label{rmk-exp-small}
The relative position of the S-J-P-A can be described by the Figure \ref{fig-4bp}. 
From this Figure we can get an underlying restriction $0<\dt\lesssim \cO(\exp(-\eps_{\dt}^{-3}))$. This is because the distance between the Asteroid and the origin has to be greater than $\cO(1/\eps_{\dt}^2)$, to avoid the collision between the Asteroid and the Planet happening. 
That implies the angular momentum $\Psi$ of the Asteroid should be greater than $\cO(1/\eps_{\dt})$. 
However, in \cite{GMS} it is proved that for the RPC3BP, the splitting between the manifolds of the infinity $\Lb_{[\Psi_1,\Psi_2]}^{\dt,\mu}$  won't exceed $\cO(\exp(-\Psi_{\max}^3))$, where $\Psi_{\max}=\Psi _2$ is the maximum value of the angular momentum. 
If we want system (\ref{ham-4bp}) to be an effective perturbation of the RPC3BP, $\dt$ has to be imposed an upper restriction $0<\dt\le \cO(\exp(-1/\eps_{\dt}^{3}))$.\\
\end{rmk}

\subsection{Shadowing orbits in the PR4BP}\label{s4}
\vspace{20pt}

Based on the transversality of $W^u(\Lb_{[\Psi_1,\Psi_2]}^{\dt,\mu})$ and $W^s(\Lb_{[\Psi_1,\Psi_2]}^{\dt,\mu})$ proved in Sec. \ref{s3}, we want to obtain the existence of shadowing orbits along the obtained infinite transition chain of the scattering map through a suitable $\lambda-$Lemma. 
As the manifold $\Lb_{[\Psi_1,\Psi_2]}^{\dt,\mu}$ if parabolic, we will apply  the $\lambda-$Lemma in \cite{GMSS} which can be easily adapted to our system.
In fact,  as we can see from the proof of Theorem \ref{stable-mani-thm}
in Sec. \ref{s3}, we have showed that near infinity system \eqref{eq-infity}  in the coordinates $(q,p,\alpha,\Psi,s)$ given by \eqref{eq:changepq},  becomes system \eqref{eq:systempq}, which is analogous to system (14) in \cite{GMSS}. 
Therefore, the Lambda lemma for this system established in  
that paper immediately gives the 
following $\lb-$Lemma:

\begin{lem}[$\lb-$Lemma]\label{lam-lem}
Let $\gamma$ be a curve which transversally intersects 
$W^s(\Lb_{[\Psi_1,\Psi_2]}^{\dt,\mu})$ at a point $P\in W^s(\Lb_{\psi_0,\Psi_0}^{\dt,\mu})$ for some fixed point 
$\Lb_{\psi_0,\Psi_0}^{\dt,\mu}\in \Lb_{[\Psi_1,\Psi_2]}^{\dt,\mu}$. Let $Z\in W^u(\Lb_{\psi_0,\Psi_0}^{\dt,\mu})$ be another point. 
For any neighborhood $\cU$ of $Z$ in $\R^4$ and any $\epsilon>0$, there exists a point $a\in B_\epsilon(P)\cap\gamma$ and a positive integer $n$ depending on $Z,\epsilon, \cU$ such that $\cP^n(a)\in\cU$. 
As a consequence $W^u(\Lb_{\psi_0,\Psi_0}^{\dt,\mu})\subset\overline{\cup_{j\geq0}\cP^j(\Gamma)}$.
\end{lem}
\begin{rmk}
Since system (\ref{ham-4bp}) is reversible of time $t$, we can get a similar conclusion by reversing the time. 
\end{rmk}
Benefit from Lemma \ref{lam-lem}, now we give the shadowing result which gives the 
existence of shadowing orbits, by a standard argument proved in  \cite{DLS}. We omit the proof here, because is done in \cite{DLS} in the hyperbolic case and adapted in \cite{GMSS} for the parabolic one:


\begin{prop}[Shadowing orbits]\label{shad-prop}
Let $\{\Lb_{\psi_k,\Psi_k}^{\dt,\mu}\}_{k\in\mathbb N}$ be a family of parabolic fixed points in $\Lb^{[\Psi_1,\Psi_2]}_{\dt,\mu}$ of the Poincar\'e map $\cP$ in (\ref{poin-4bp}), such that for all $k\in\mathbb N$, $W^u(\Lb_{\psi_k,\Psi_k}^{\dt,\mu})$ intersects $W^s(\Lb_{[\Psi_1,\Psi_2]}^{\dt,\mu})$ transversally at $P_k\in W^s(\Lb_{\psi_{k+1},\Psi_{k+1}}^{\dt,\mu})$. 
Accordingly, for any two sequences of real numbers $\{\iota_k\}_{k\in\mathbb N}$ and $\{\wt\iota_k\}_{k\in\mathbb N}$ with $0<\iota_k,\wt\iota_k\ll 1$ sufficiently small, there exist $a\in B_{\iota_0}(\Lb_{\psi_0,\Psi_0}^{\dt,\mu})$ and two sequences of positive integers $\{N_k\}_{k\in\mathbb N}$, $\{\wt N_k\}_{k\in\mathbb N}$ satisfying $N_k<\wt N_k<N_{k+1}<\wt N_{k+1}$ for all $k$, such that 
\[
dist(\cP^{N_k}(a),\Lb_{\psi_k,\Psi_k}^{\dt,\mu})\leq \iota_k,\quad dist(\cP^{\wt N_k}(a), P_k)\leq \wt\iota_k
\]
 for all $k\in\mathbb N$, see Fig. \ref{} for a concrete impression.
\end{prop}

\begin{figure}
\begin{center}
\includegraphics[width=14cm]{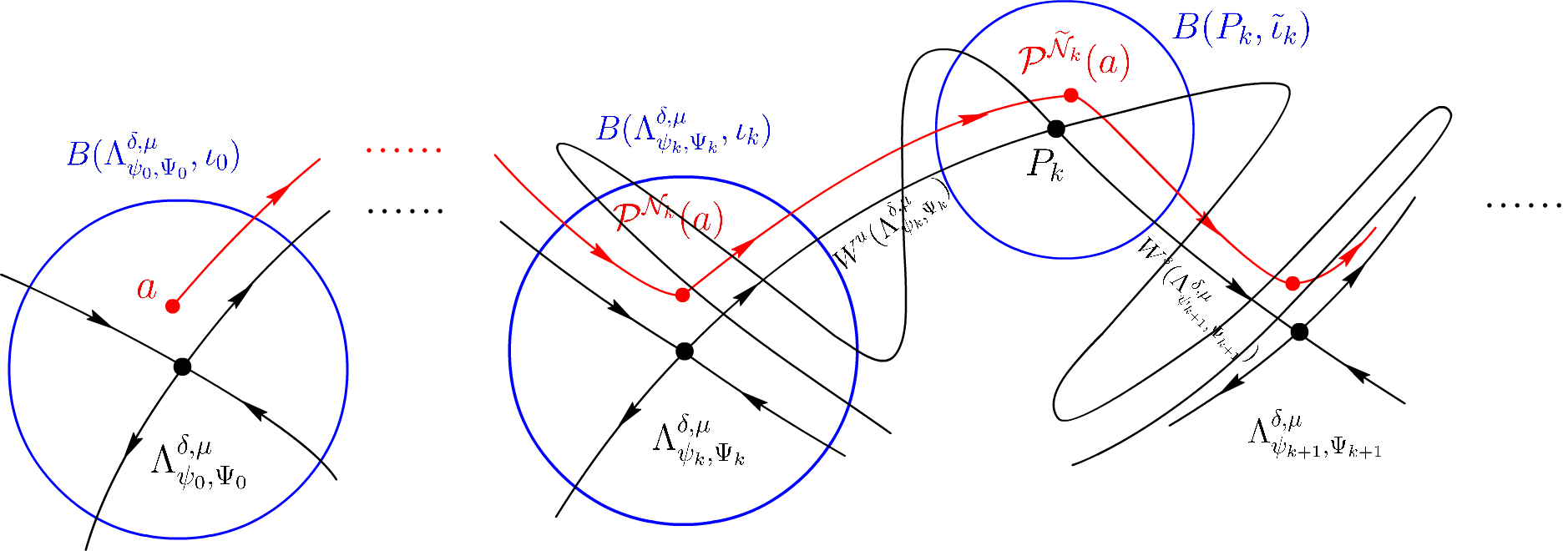}
\caption{shadowing orbits and the $\lb-$Lemma }
\label{}
\end{center}
\end{figure}

Now we can derive the second item of  Theorem \ref{main-thm} directly from this Proposition:

\begin{proof}{\it of Theorem \ref{main-thm}}. 
Let $\{\Lb_{\psi_k,\Psi_k}^{\dt,\mu}\}_{k\in\N}$ be one of the bounded orbits given in subsection \ref{s3s2} for the scattering map $\cS^+$: 
$\Lb_{\psi_k,\Psi_k}^{\dt,\mu}=\cS^+(\Lb_{\psi_{k-1},\Psi_{k-1}}^{\dt,\mu})$ for all $k\geq 1$.
Applying Proposition \ref{shad-prop}, we take $\iota_k=1/k$ and $\wt \iota_k=\wt\iota$ uniformly small such that $B_{\wt\iota_k}(P_k)$ don't intersect $\Lb^{[\Psi_1,\Psi_2]}_{\dt,\mu}$. There exists integers $N_k$ and $\wt N_k$ due to Prop \ref{shad-prop}, such that for some $a\in B_{\iota_0}(\Lb_{\psi_0,\Psi_0}^{\dt,\mu})$, dist$(\cP^{N_k}(a),\Lb_{\psi_k,\Psi_k}^{\dt,\mu})\leq 1/k$ and dist$(\cP^{\wt N_k}(a),\Lb_{\psi_k,\Psi_k}^{\dt,\mu})\leq \wt\iota$ for all $k\in\Z_+$.  That implies  the orbit starting from $a$ is oscillatory, since
$\iota_k\rightarrow 0$ as $k\rightarrow+\infty$ and $B_{\wt\iota}(P_k)$ doesn't intersect $\{x=\Xi=0\}$. 
\end{proof}
\vspace{20pt}

\appendix

\section{Rescaling transformations for the 3BP}\label{A}
\vspace{20pt}
In this appendix we give some more details about the transformations done to system 
$H_{rot}$ in (\ref{polar sys-4bp}) and the RPC3BP. 
  Nonetheless, we fix $\Omega=\alpha$, take  $\dt\ll\mu$ and make the following rescaling 
       \be\label{trans-res}
 \Phi_{res}^1:  \left\{
   \begin{aligned}
   p_2&=\dt v_2,&\\
   r&=1+\sqrt\dt\wt r,& \\
   R&=\sqrt\dt\wt R,& 
   \end{aligned}
   \right.
   \ee
   then 
   \[
   dr\wedge dR+d\theta\wedge d\Om+dq_2\wedge dp_2=\dt (d\wt r\wedge d\wt R+dq_2\wedge dv_2).
   \]
Therefore, we obtain a Hamiltonian system of Hamiltonian:
  \be\label{rescal-sys-4bp}
  \wt H_{res}^\dt(\wt r,\wt R,q_2,v_2):=\frac{H_{rot}+\alpha/2}{\dt},
  \ee
   that can be expressed by
   \be
    \wt H_{res}^\dt(\wt r,\wt R, q_2,v_2)&=&\underbrace{\Big[\frac12|v_2|^2-\frac{1-\mu}{|q_2+ (\mu,0)|}-\frac{\mu}{|q_2-(1-\mu,0)|}-q_2\times v_2\Big]}_{RPC3BP}\nonumber\\
  & &+\underbrace{\frac12\Big[\frac{\wt R^2}{\alpha}+\alpha\wt r^2\Big]}_{rotator}+\Delta \wt H_{res}^\dt(\wt r,q_2, v_2) \label{eq:Hres}
  \ee
  where
  \begin{equation}\label{eq:Hresdelta}
\Delta \wt H_{res}^\dt(\wt r,q_2, v_2)=\frac{\dt}{2}|v_2|^2+ \wt f^\dt(\wt r, q_2, v_2)+ \wt g^\dt(\wt r,q_2)
      \end{equation}
  with: 
  \be\label{f-fun}
\wt f^\dt(\wt r, q_2, v_2)&=& \frac{\alpha^2\wt r^2-2\alpha q_2\times v_2+\dt(q_2\times v_2)^2}{2\alpha(1+\sqrt\dt\wt r)^2}+q_2\times v_2-\frac{\alpha\wt r^2}{2}\nonumber\\
 &=&\frac{\dt(q_2\times v_2)^2+(2\sqrt\dt\wt r+\dt\wt r^2)\cdot(2\alpha q_2\times v_2-\alpha^2\wt r^2)}{2\alpha(1+\sqrt\dt\wt r)^2}
  \ee
  and
  \be\label{g-fun}
   \wt g^\dt(\wt r,q_2)&=&\Big[\frac{1-\mu}{|q_2+ (\mu,0)|}-\frac{1-\mu}{|q_2+ \mu(1+\sqrt\dt\wt r) (1,0)|}\Big]\nonumber\\
    & &+\Big[\frac{\mu}{|q_2-(1-\mu,0)|}-\frac{\mu}{|q_2-(1-\mu)(1+\sqrt\dt\wt r)(1,0)|}\Big].
  \ee
\begin{rmk}
Taking $\dt\in[0,\mu)$ as a parameter, then for $\dt=0$ the system (\ref{rescal-sys-4bp}) becomes a direct sum of a RPC3BP 
system and a rotator; 
Moreover, as $\dt\rightarrow 0$, $\Dt \wt H_{res}^\dt\rightarrow 0$,  
in fact $\Dt \wt H_{res}^\dt=\mathcal{O}(\sqrt{\dt})$ when the variables are bounded.
Therefore, for sufficiently small $\dt$, we can apply the perturbative theory to (\ref{rescal-sys-4bp}) and show the persistence of certain periodic orbits for the RPC3BP.\\
\end{rmk}

As is said in Section \ref{s2}, we try to seek the comet-type periodic orbits for $H_{rot}^\dt$ in (\ref{polar sys-4bp}). 
Aiming this, we need transfer $\wt H^\dt_{res}$ further, until we get the desired system. 
Precisely,  for $|q_2|\gg1$, we have the estimate
\be
\wt H_{res}^\dt(\wt r, \wt R,q_2,v_2)&=&\Big[\frac12|v_2|^2-q_2\times v_2-\frac{1}{|q_2|}+\cO(\frac{\mu}{|q_2|^3})\Big]+
\frac12\Big[\frac{\wt R^2}{\alpha}+\alpha\wt r^2\Big]\nonumber\\
& &+\Dt \wt H_{res}^\dt(\wt r, q_2,v_2);\nonumber
\ee
If we apply a further step rescaling, i.e., we take a number $0<\eps \ll 1$ and we define:
\be\label{trans-res-2}
 \Phi_{res}^2:\ q_2=\frac{\wh q_2}{\eps^2},\quad  v_2=\eps \wh v_2,\quad \wt r=\frac{\wh r}{\sqrt\eps},\quad\wt R=\frac{\wh R}{\sqrt\eps},\quad
 0<\eps\ll1,
\ee
then 
\[
d\wt r\wedge d\wt R+dq_2\wedge dv_2=\frac1\eps(d\wh r\wedge d\wh R+d\wh q_2\wedge d\wh v_2).
\]
Consequently the new system is Hamiltonian with  Hamiltonian
\[
\wh H_{res}^{\dt,\eps} (\wh r, \wh R,\wh q_2,\wh v_2):=
\eps \cdot \wt H_{res}^\dt(\frac{\wh r}{\sqrt\eps},\frac{\wh R}{\sqrt\eps},\frac{\wh q_2}{\eps^2},\eps \wh v_2),
\]
its flow preserves the symplectic form $d\wh r\wedge d\wh R+d\wh q_2\wedge d\wh v_2$. 
 Moreover, $\wh H_{res}^{\dt,\eps}$ has the following expression:
\be\label{rescal-sys-4bp-2}
\wh H_{res}^{\dt,\eps}(\wh r, \wh R,\wh q_2,\wh v_2) &=&\Big[-\wh q_2\times \wh v_2+\eps^3(\frac{|\wh v_2|^2}2-\frac{1}{|\wh q_2|})+\cO(\frac{\mu\eps^7}{|\wh q_2|^3})\Big]\nonumber\\
& &+\frac12\Big[\frac{\wh R^2}{\alpha}+\alpha\wh r^2\Big]+
\Dt \wh H_{res}^\dt(\wh r, \wh q_2,\wh v_2),
\ee
where:
 \[
\Dt \wh H_{res}^{\dt,\eps}(\wh r, \wh q_2, \wh v_2)
 =\frac{\dt \eps ^3}{2}|\wh v_2|^2+ 
\wh f^\dt(\wh r, \wh q_2, \wh v_2)+ \wh g^\dt(\wh r,\wh q_2)
      \]
  with: 
  \be\label{f-fun1}
 \wh f^\dt(\wh r, \wh q_2, \wh v_2)&=& 
  \frac{ \frac{\dt}{\eps}(\wh q_2\times\wh v_2)^2+(2\sqrt\frac{\dt}{\eps}\wh r+\frac{\dt}{\eps}\wh r^2)
  \cdot(2\alpha \wh q_2\times \wh v_2-\alpha^2\wh r^2)}{2\alpha(1+\sqrt\frac{\dt}{\eps}\wh r)^2}
  \ee
  and
  \be\label{g-fun1}
  \wh g^\dt(\wh r,\wh q_2)&=&\eps ^3\Big[\frac{1-\mu}{|\wh q_2+ \mu \eps ^2(1,0)|}-
 \frac{1-\mu}{|\wh q_2+ \mu\eps ^2(1+\frac{\sqrt\dt}{\sqrt\eps}\wh r) (1,0)|}\Big]\nonumber\\
    & &+\eps ^3\Big[\frac{\mu}{|\wh q_2-\eps ^2(1-\mu)(1,0)|}-\frac{\mu}{|\wh q_2-\eps ^2(1-\mu)(1+\sqrt{\frac{\dt}{\eps}}\wh r)(1,0)|}\Big].
  \ee

For convenience, we can further transfer the system to the polar coordinate, i.e.
\[
(\wh q_2,\wh v_2)\xrightarrow{\Phi_{pol}} (\rho,\phi,\Upsilon,G), \quad\text{via } \left\{
\begin{split}
\pi_1\wh q_2&=\rho\cos\phi, \\
\pi_2\wh q_2&=\rho\sin\phi,\\
\pi_1\wh v_2&=\Upsilon\cos\phi-\frac{G}{\rho}\sin\phi, \\
\pi_2\wh v_2&=\Upsilon\sin\phi+\frac{G}{\rho}\cos\phi,
\end{split}
\right.
\]
for $(\wh r,\wh R,\rho,\phi,\Upsilon,G)\in\cD_{res}$ (see (\ref{eq:Dres})), then we get 
\be\label{res-sys-comet}
H_{res}^{\dt,\eps}(\wh r,\wh R,\rho,\Upsilon,\phi,G)&=&-G+\eps^3\Big[\frac12(\Upsilon^2+\frac{G^2}{\rho^2})-\frac1\rho\Big]+\frac12\Big[\frac{\wh R^2}{\alpha}+\alpha\wh r^2\Big]\nonumber\\
& &+\cO(\frac{\mu\eps^7}{|\rho|^3})+\Dt H_{res}^\dt(\wh r,\rho,\phi,\Upsilon,G),
\ee
where
\be\label{exp-R}
\Dt H_{res}^{\dt,\eps} (\wh r,\rho,\phi,\Upsilon,G)=f^\dt(\wh r, G)+g^\dt(\wh r,\rho,\phi)+\frac{\dt\eps^3} 2(\Upsilon^2+\frac{G^2}{\rho^2})
\ee
with
\be\label{exp-f}
f^\dt(\wh r, G)= \frac{ \frac{\dt}{\eps}G^2+(2\sqrt\frac{\dt}{\eps}\wh r+\frac{\dt}{\eps}\wh r^2)
  \cdot(2\alpha G-\alpha^2\wh r^2)}{2\alpha(1+\sqrt\frac{\dt}{\eps}\wh r)^2}=\cO(\sqrt{\frac{\dt}{\eps}})
\ee
and
\be\label{exp-g}
g^\dt(\wh r,\rho,\phi)&=& \mu(1-\mu)\sqrt{\frac{\dt}{\eps}}\wh r\eps^5\Big\{\frac{\rho\cos\phi+\mu\eps^2}{\big[\rho^2+2\eps^2\mu\rho\cos\phi+\eps^4\mu^2\big]^{3/2}}\nonumber\\
& &+\frac{\rho\cos\phi+(\mu-1)\eps^2}{\big[\rho^2-2\eps^2(1-\mu)\rho\cos\phi+\eps^4(1-\mu)^2\big]^{3/2}}\Big\}+\cO(\dt\eps^5).
\ee
Therefore, 
\be
\big\|\Dt H_{res}^{\dt,\eps}(\wh r,\rho,\phi,\Upsilon,G)\big\|_{C^2}\lesssim \sqrt{\frac{\dt}{\eps}}
\ee
as long as $\dt\lesssim o(\eps)$.

\vspace{50pt}

%
%
%

\bibliographystyle{abbrv}
\bibliography{Oscillatory}
\end{document}